\documentclass[12pt]{article}

\usepackage{latexsym,amsmath,amscd,amssymb,framed,graphics,color}
\usepackage{enumerate}

\usepackage{graphicx}

\usepackage{url}

\usepackage[all]{xy}

\newcommand{\bfi}{\bfseries\itshape}
\newcommand{\rem}[1]{}

\makeatletter

\@addtoreset{figure}{section}
\def\thefigure{\thesection.\@arabic\c@figure}
\def\fps@figure{h, t}
\@addtoreset{table}{bsection}
\def\thetable{\thesection.\@arabic\c@table}
\def\fps@table{h, t}
\@addtoreset{equation}{section}

\makeatother

\newcommand \al{\alpha}
\newcommand\be{\beta}

\newcommand\de{\delta}

\newcommand\et{\eta}
\renewcommand\th{\theta}

\newcommand\la{\lambda}

\newcommand\si{\sigma}

\newcommand\ph{\varphi}

\newcommand\ps{\psi}
\newcommand\om{\omega}
\newcommand\Ga{\Gamma}

\newcommand\La{\Lambda}

\newcommand\Ph{\Phi}
\newcommand\Ps{\Psi}
\newcommand\Om{\Omega}

\newcommand\ie{i.e.\ }

\renewcommand\o{\circ}

\newcommand\x{\times}
\newcommand\on{\operatorname}
\newcommand\Ad{\on{Ad}}

\newcommand\Emb{\on{Emb}}
\newcommand\ev{\on{ev}}
\newcommand\Den{\on{Den}}
\newcommand\pr{\on{pr}}

\newcommand\F{\mathcal{F}}
\newcommand\Diff{\on{Diff}}

\newcommand\g{\mathfrak g}

\newcommand\dd{{\mathbf d}}
\newcommand\ii{{\mathbf i}}

\newcommand\J{\mathbf{J}}
\newcommand\X{\mathfrak X}
\renewcommand\P{\mathcal P}

\newcommand\ZZ{\mathbb Z}
\newcommand\RR{\mathbb R}
\newenvironment{proof}[1][Proof]{\noindent\textbf{#1.} }{\ \rule{0.5em}{0.5em}}

\newcommand{\fint}{-\!\!\!\!\!\!\int}

\def\XXint#1#2#3{{\setbox0=\hbox{$#1{#2#3}{\int}$ }
\vcenter{\hbox{$#2#3$ }}\kern-.5\wd0}}

\begin{document}

\newtheorem{theorem}{Theorem}[section]
\newtheorem{definition}[theorem]{Definition}
\newtheorem{lemma}[theorem]{Lemma}
\newtheorem{remark}[theorem]{Remark}
\newtheorem{proposition}[theorem]{Proposition}
\newtheorem{corollary}[theorem]{Corollary}
\newtheorem{example}[theorem]{Example}

\def\below#1#2{\mathrel{\mathop{#1}\limits_{#2}}}



\title{Dual pairs in fluid dynamics}
\author{Fran\c{c}ois Gay-Balmaz$^{1,2}$ and Cornelia Vizman$^{3}$ }

\addtocounter{footnote}{1}
\footnotetext{Control and Dynamical Systems, California Institute of Technology 107-81, Pasadena, CA 91125, USA.
\texttt{fgbalmaz@cds.caltech.edu}
\addtocounter{footnote}{1} }
\footnotetext{LMD, Ecole Normale Sup\'erieure/CNRS, Paris, France.
\texttt{gaybalma@lmd.ens.fr}
\addtocounter{footnote}{1} }
\footnotetext{Department of Mathematics,
West University of Timi\c soara, RO--300223 Timi\c soara, Romania.
\texttt{vizman@math.uvt.ro}
\addtocounter{footnote}{1} }

\date{ }
\maketitle

\makeatother
\maketitle




\begin{abstract}
This paper is a rigorous study of the dual pair structure of the ideal fluid (\cite{MaWe83}) and the dual pair structure for the $n$-dimensional Camassa-Holm (EPDiff) equation (\cite{HoMa2004}), including the proofs of the necessary transitivity results. In the case of the ideal fluid, we show that a careful definition of the momentum maps leads naturally to central extensions of diffeomorphism groups such as the group of quantomorphisms and the Ismagilov central extension.
\end{abstract}

\tableofcontents



\section{Introduction}

The concept of dual pair, formalized by \cite{We83}, is an important notion in Poisson geometry and has many applications in the context of
momentum maps and reduction theory, see e.g. \cite{OrRa04} and references therein. Let $(M,\omega)$ be a finite dimensional symplectic manifold and let $P_1, P_2$ be two finite dimensional Poisson manifolds. A pair of Poisson mappings
\begin{equation*}
P_1\stackrel{\J_1}{\longleftarrow}(M,\om)\stackrel{\J_2}{\longrightarrow} P_2
\end{equation*}
is called a \textit{dual pair\/} if $\ker T\J_1$ and $\ker T\J_2$ are symplectic orthogonal complements of one another, where $\ker T\J_i$ denotes the kernel of the tangent map $T\mathbf{J}_i$ of $\mathbf{J}_i$.
Dual pair structures arise naturally in classical mechanics. In many cases, the Poisson maps $\mathbf{J}_i$ are momentum mappings associated to Lie algebra actions on $M$.
For example, in \cite{Ma1987} (see also \cite{CuRo1982}, \cite{GoSt1987} 
and \cite{Iw1985}) it was shown that the concept of dual pair of momentum maps can be useful for the study of bifurcations in Hamiltonian systems with symmetry.

We now describe two fundamental ``dual pairs'' of momentum mappings arising in fluid dynamics and that we shall study in detail in this paper. Both of them have very attractive properties but present additional difficulties due to fact that the manifolds involved are infinite dimensional. This is why we describe them only at a formal level in this Introduction and we use the word ``dual pair" (with quotation mark) in a very formal sense. 
The first ``dual pair" is associated to the Euler equations of an ideal fluid and was discovered by \cite{MaWe83} in the context of Clebsch variables; the second one is associated to the $n$-dimensional Camassa-Holm equations and was discovered by \cite{HoMa2004} in the context of singular solutions.

Consider the Euler equations for an ideal fluid on a domain $S$,
\[
\partial_t u+u\cdot\nabla u=-\operatorname{grad}p,\quad\operatorname{div}u=0,
\]
where $u$ is the velocity and $p$ is the pressure. As shown in \cite{Ar1966}, the flow of the Euler equations describe geodesics on the group of volume preserving diffeomorphisms of $S$ relative to the right invariant $L^2$-metric. In \cite{MaWe83}, the ``dual pair'' for Euler equations is described as follows. Consider a symplectic manifold $(M,\omega)$, a volume manifold $(S,\mu)$, and let $\mathcal{F}(S,M)$ be the space of smooth maps from $S$ to $M$. The left action of the group $\operatorname{Diff}(M,\omega)$ of symplectic diffeomorphisms and the
right action of the group $\operatorname{Diff}(S,\mu)$ of volume preserving diffeomorphisms are two commuting symplectic actions on $\F(S,M)$. Their momentum maps $\mathbf{J}_L$ and $\mathbf{J}_R$
form the ``dual pair" for the Euler equation:

\begin{picture}(150,100)(-70,0)%
\put(110,75){$\F(S,M)$}

\put(90,50){$\mathbf{J}_L$}

\put(160,50){$\mathbf{J}_R$}

\put(52,15){$\X(M,\om)^*$}

\put(170,15){$\X(S,\mu)^*$.}

\put(130,70){\vector(-1, -1){40}}

\put(135,70){\vector(1,-1){40}}

\end{picture}\\
While the right leg represents Clebsch variables for the Euler equations, the left leg is a constant of motion for the induced Hamiltonian system on $\F(S,M)$.

Consider the $n$-dimensional Camassa-Holm equations on a domain $M$,
\[
\partial_tm+u\cdot\nabla m+\nabla u^\mathsf{T}\cdot m+m\operatorname{div} u=0,\quad m=(1-\alpha^2\Delta)u,
\]
whose flow describes geodesics on the group of all diffeomorphisms of $M$ relative to a right invariant $H^1$ metric, see \cite{HoMa2004}. For more general choices for $m$, these equations are known under the generic name of EPDiff equations (standing for the Euler-Poincar\'e equations associated with the diffeomorphism group). In \cite{HoMa2004}, the associated dual pair is described as follows. Let $\Emb(S,M)$ be the space of embeddings of $S$ into $M$ and consider the left action of the diffeomorphism group $\Diff(M)$ and the right action of the diffeomorphism group $\Diff(S)$. The dual pair consists of the momentum maps associated to the induced actions on the cotangent bundle $T^*\Emb(S,M)$ endowed with canonical symplectic form:

\begin{picture}(150,100)(-70,0)%
\put(100,75){$T^*\Emb(S,M)$}

\put(90,50){$\mathbf{J}_L$}

\put(160,50){$\mathbf{J}_R$}

\put(65,15){$\mathfrak{X}(M)^{\ast}$}

\put(170,15){$\mathfrak{X}(S)^{\ast}.$}

\put(130,70){\vector(-1, -1){40}}

\put(135,70){\vector(1,-1){40}}

\end{picture}\\
The left leg provides singular solutions of the EPDiff equation, whereas the right leg is a constant of motion associated to the collective motion on $T^*\operatorname{Emb}(S,M)$. In the one-dimensional case, $\mathbf{J}_L$ recovers the peakon solutions of the one-dimensional Camassa-Holm equation, \cite{CaHo1993}. In the particular case $S=M$, the left leg recovers the reduction map from Lagrangian to Eulerian variables.

More recently, the ideal fluid ``dual pair" has been shown to apply for the Vlasov equation in kinetic theory \cite{HoTr2009}. On the other hand, the EPDiff dual pair has been extended to the case of the Euler-Poincar\'e equations associated with the automorphism group of a principal bundle in \cite{GBTV10}, needed for the study of the singular solutions of the two-component Camassa-Holm equation and its generalizations, see \cite{HoTr2008}, \cite{GBTV10}, and references therein.

\medskip

As we mentioned above, the reader should be warned that we use here the word ``dual pair" (with quotation mark) in a formal sense for many reasons. Firstly, these examples are infinite dimensional and the definition of the concept of dual pair in infinite dimensions presents several difficult points that deserve further investigation. Secondly, the dual pair properties need to be shown in a rigorous way for the two situations. 
This means that one has to prove that the left action is transitive on the level subset of the right momentum map, and the right action is transitive on the level subset of the left momentum map.
Finally, the ideal fluid ``dual pair" reveals additional difficulties when one wants to check the momentum map properties of its two legs in a rigorous way. In fact, a reformulation is needed in this case. Interestingly enough, this reformulation leads naturally to well-known central extensions of groups of diffeomorphisms. The main goal of the paper is to overcome these difficulties in order to rigorously show the dual pair properties.

\paragraph{Plan and results of the paper.} In Section \ref{deux} below we review some basic facts concerning dual pairs in finite dimensions and mention several difficulties appearing in the infinite dimensional case. We define the concepts of weak dual pairs and dual pairs, appropriate in the infinite dimensional setting, and give criteria for weak dual pair and dual pair properties. The main goal of Section \ref{trois} is to provide a reformulation of the ideal fluid dual pair that allows us to show in a rigorous way that the two legs are momentum mappings. More precisely, in a first step we replace the groups $\operatorname{Diff}(M,\omega)$ and $\operatorname{Diff}(S,\mu)$ by the subgroups $\operatorname{Diff}_{ham}(M,\omega)$ and $\operatorname{Diff}_{ex}(S,\mu)$ of Hamiltonian and exact volume preserving diffeomorphisms, respectively. This allows us to show the existence of nonequivariant momentum maps. In order to have equivariance, required from the dual pair properties, we need to consider central extensions. In this context the group of quantomorphisms and the Ismagilov central extension of the group $\operatorname{Diff}_{ex}(S,\mu)$ appear naturally. Section \ref{quatre} deals with the dual pair properties for the ideal fluid case. We show that the pair of momentum maps obtained in Section \ref{trois} form a weak dual pair. When restricted to the space of embeddings and under the condition $H^1(S)=0$, the weak dual pair is shown to be a dual pair. In Section \ref{cinq} we prove that the pair of momentum maps associated to the EPDiff equation form a weak dual pair. Then we show that by restricting these momentum maps to the open subset $T^*\operatorname{Emb}(S,M)^\times$ of $1$-form densities which are everywhere non-zero on $S$, we obtain a dual pair. Finally, in Section \ref{six} we construct a natural symplectic map that relates both the Euler and the EPDiff dual pair constructions.

\medskip

Throughout the paper we use the hat calculus for differential forms on the Fr\'echet manifold $\F(S,M)$ which are induced from differential forms on $S$ and $M$.
It was developed in
\cite{Vi09} and its properties are presented in the Appendix to this paper.

\paragraph{Acknowledgements.}
We are grateful to Stefan Haller for very helpful suggestions,
especially regarding Lemma \ref{hard}. We thank Cesare Tronci for many useful and pleasant conversations about these and related matters.


\section{Dual pairs in infinite dimensions}\label{deux}

After giving general definitions and introducing the possible obstacles in
the infinite dimensional framework, this section formulates the concept of
weak dual pair and prepares the results that will be applied to the Euler and
the EPDiff dual pairs. 

\subsection*{Finite dimensional case}\label{finite}

Let $(M,\omega)$ be a \textit{finite dimensional} symplectic manifold and $P_1, P_2$ be two finite dimensional Poisson manifolds. A pair of Poisson mappings
\begin{equation}\label{pois}
P_1\stackrel{\J_1}{\longleftarrow}(M,\om)\stackrel{\J_2}{\longrightarrow} P_2
\end{equation}
is called a \textit{dual pair} \cite{We83}
if $\ker T\J_1$ and $\ker T\J_2$ are symplectic orthogonal complements of one another. That is, for each $m\in M$,
\begin{equation}\label{finite_dimensional_dual_pair}
(\ker T_m\J_1)^\om=\ker T_m\J_2.
\end{equation}
The dual pair is called \textit{full} if $\J_1:M\rightarrow P_1$ and $\J_2:M\rightarrow P_2$ are surjective submersions. 
A key result in the context of dual pairs is the correspondence of
symplectic leaves. Suppose that \eqref{pois} is a full dual pair and
that $\J_1$, $\J_2$ have connected fibers. Then there is a bijective
correspondence between the symplectic leaves 
of $P_1$ and those 
of $P_2$, given by
$\mathcal{L}_1\mapsto \J_2(\J_1^{-1}(\mathcal{L}_1))$, with inverse
$\mathcal{L}_2\mapsto \J_1(\J_2^{-1}(\mathcal{L}_2))$ \cite{We83}.
We refer to Chapter 11 in \cite{OrRa04} for further informations on dual pairs.

\paragraph{Two basic finite dimensional examples.} We now give two basic examples of dual pairs. We suppose for the moment that all the manifolds involved are finite dimensional.
\begin{itemize}
\item[(\bf i)] Let $G$ be a Lie group acting canonically (on the left) on a symplectic manifold $(M,\omega)$ and admitting a momentum map $\mathbf{J}:M\rightarrow\mathfrak{g}^*$.
This means that $\mathbf{d}\langle\mathbf{J},\xi\rangle=\mathbf{i}_{\xi_M}\omega$ for all 
$\xi\in \mathfrak{g}$, where $\xi_M$ denotes the infinitesimal generator.
The kernel and the image of the tangent map $T\mathbf{J}$ are characterized by the equalities
\begin{equation}\label{ker_range}
\operatorname{ker}T_m\mathbf{J}=\mathfrak{g}_M(m)^\omega\quad\text{and}
\quad\operatorname{range}T_m\mathbf{J}=\left(\mathfrak{g}_m\right)^\circ,
\end{equation}
where $\mathfrak{g}_M(m):=\{\xi_M(m)\mid \xi\in\mathfrak{g}\}$ and
$\left(\mathfrak{g}_m\right)^\circ=\{\mu\in\mathfrak{g}^* \mid
\langle\mu,\xi\rangle=0,\,\forall\xi\in\mathfrak{g}_m\}$ denotes the
annihilator in $\mathfrak{g}^*$ of the isotropy subalgebra
$\mathfrak{g}_m$ of $m$. This shows that, when the action is free,
$\mathbf{J}$ is a submersion onto an open subset of
$\mathfrak{g}^*$. As recalled in the Appendix, 
when the momentum map is equivariant, it is a Poisson map with respect to the symplectic form on $M$ 
and the $(+)$ Lie-Poisson structure
on $\mathfrak{g}^*$. If we suppose in addition
that $G$ acts freely and properly on $M$, then the quotient space
$M/G$ is a smooth manifold such that the projection
$\pi:M\rightarrow M/G$ is a smooth surjective submersion. This map
is Poisson with respect to the symplectic form on $M$ and the
induced quotient Poisson structure on $M/G$. Thus, using the
equality
\[
\operatorname{ker}T\mathbf{J}=\left(\mathfrak{g}_M\right)^\omega=(\operatorname{ker}T\pi)^\omega,
\]
we obtain the dual pair  (\cite{We83})
\begin{equation}\label{reduction}
M/G\stackrel{\pi}{\longleftarrow}(M,\om)\stackrel{\mathbf{J}}{\longrightarrow} \mathfrak{g}^*.
\end{equation}
By replacing $\mathfrak{g}^*$ by its open subset $\mathbf{J}(M)$, we get a full dual pair.

\item[(\bf ii)] The example we consider here is important since the dual pairs we shall study in this paper are  infinite dimensional versions of it.
We first treat the case of Lie algebra actions.

Let $\mathfrak{g}_1, \mathfrak{g}_2$ be two Lie algebras acting canonically on the symplectic manifold $(M,\omega)$ and admitting the infinitesimally equivariant momentum maps $\mathbf{J}_1$ and $\mathbf{J}_2$. Suppose that $\mathbf{J}_1$ is infinitesimally invariant under the action of $\mathfrak{g}_2$, that is $(\mathfrak{g}_2)_M\subset \operatorname{ker}(T\mathbf{J}_1)$. By \eqref{ker_range} this is equivalent to the assumption that $\mathbf{J}_2$ is infinitesimally invariant under the action of $\mathfrak{g}_1$. Note also that these conditions are equivalent to the statement $\omega(\xi_M,\eta_M)=0$, for all $\xi\in\mathfrak{g}_1, \eta\in\mathfrak{g}_2$, that is, $(\mathfrak{g}_1)_M\subset \left((\mathfrak{g}_2)_M\right)^\omega$. In this case, the Lie algebra actions commute since $\left[\xi_M,\eta_M\right]=- X_{\omega(\xi_M,\eta_M)}=0$, where, as usual, $X_h$ denotes the Hamiltonian vector field associated to the function $h$ on $M$.
In order to obtain a dual pair, we need to assume that $\mathfrak{g}_2$ acts transitively on the level set of $\mathbf{J}_1$, that is $(\mathfrak{g}_2)_M=\operatorname{ker}T\mathbf{J}_1$. 
By \eqref{ker_range}, this is equivalent to assume that $\mathfrak{g}_1$ acts transitively on the level set of $\mathbf{J}_2$. In this case
\[
\mathfrak{g}_1^*\stackrel{\mathbf{J}_1}{\longleftarrow}(M,\om)\stackrel{\mathbf{J}_2}{\longrightarrow} \mathfrak{g}^*_2
\]
form a dual pair.

We now suppose that the Lie algebra actions are associated to symplectic Lie group actions of $G_1$ and $G_2$. Suppose that $\mathbf{J}_1$ is $G_2$-invariant. Then we have the infinitesimal invariances $(\mathfrak{g}_1)_M\subset \operatorname{ker}T\mathbf{J}_2$ and $(\mathfrak{g}_2)_M\subset \operatorname{ker}T\mathbf{J}_1$. If $G_1$ is connected, then $\mathbf{J}_2$ is also $G_1$-invariant. If one of the actions is transitive on the corresponding level set, then the inclusions become equalities, and we get a dual pair.
Note that when $M$ is connected the other action is also transitive.

An example of this situation is provided by the case $M=T^*Q$, where the group actions are lifted from commuting actions on $Q$. In this case, it suffices to verify the transitivity hypothesis to obtain a dual pair.
\end{itemize}

\subsection*{Infinite dimensional case}

When trying to define a useful concept of dual pair on an infinite
dimensional manifold, one is faced to several difficulties. First of
all, the examples of Poisson manifolds we will consider are not
strictly speaking Poisson manifolds, the bracket being defined only
on a subalgebra of the smooth functions.
This is the case in general
for the cotangent bundle $T^*Q$ of a Fr\'echet manifold endowed with
the canonical symplectic form, where the associated bracket is
defined only on functions admitting a Hamiltonian vector field.
Therefore, in a possible definition of dual pair in infinite
dimensions, one can not expect to have Poisson maps in the usual
sense, but only in a \textit{formal} sense. Another difficulty, also
related to the weakness of the symplectic form, is that we have
inclusion $V\subset (V^\omega)^\omega$ and not equality in general,
for a subspace $V\subset T_mM$. Thus in a definition of dual pair on
an infinite dimensional manifold $M$, one has to consider Poisson
maps in a formal sense, and one has to replace
\eqref{finite_dimensional_dual_pair} by two equalities
\begin{equation}\label{strong_dual_pair}
(\ker T\J_1)^\om=\ker T\J_2,\quad (\ker T\J_2)^\om=\ker T\J_1.
\end{equation}
Such pairs of formal Poisson mappings will be called {\bfi formal dual pairs} or simply {\bfi dual pairs} in the context of infinite dimensional manifolds. Note that  if $M$ is finite dimensional these equalities are equivalent.

\medskip

We now introduce a weaker notion of dual pair, that naturally appears in the context of the ideal fluid and EPDiff equations when one wants to prove the stronger condition \eqref{strong_dual_pair}.
The pair of Poisson mappings (\ref{pois}) is called a {\bfi weak dual pair},
if $\ker T\J_1$ and $\ker T\J_2$
satisfy the inclusions
\begin{equation}\label{dual_pair}
(\ker T\J_1)^\om\subset\ker T\J_2,\quad (\ker T\J_2)^\om\subset\ker T\J_1.
\end{equation}
In finite dimensions these two inclusions are equivalent. 

\medskip

In order to be aware of the difficulties arising in the infinite
dimensional case, we reconsider the example ${\bf (i)}$ above. Let
$G$ be a connected Fr\'echet Lie group acting freely on the
symplectic Fr\'echet manifold $(M,\om)$ and admitting an equivariant
momentum map $\J:M\to\g^*$. Here and in
the following, $\mathfrak{g}^*$ denotes a topological vector space
in nondegenerate duality with the Fr\'echet Lie algebra
$\mathfrak{g}$. Standard examples for $\mathfrak{g}^*$ are the full
distributional dual or the regular dual. Note that the existence of
a momentum map $\mathbf{J}:M\to\g^*$ may depend on the chosen dual
$\mathfrak{g}^*$.

As in the finite dimensional
case, using the definition \eqref{def_momentum_map} of momentum maps, we have
\begin{equation}\label{kernel_of_TJ} \ker
T\J=(\mathfrak{g}_M)^\omega.
\end{equation}
Indeed, for a tangent vector $v_m\in T_mM$,
the assertions $T_m\J\cdot v_m=0$ and $\om(\xi_M(m),v_m)=0$ for all $\xi\in\g$
are equivalent. Assuming that the quotient space $M/G$ can be endowed with a smooth manifold structure such that the projection $\pi$ is a smooth map, we have $\mathfrak{g}_M=\operatorname{ker}(T\pi)$,
thus, using the equality \eqref{kernel_of_TJ}, we get
\begin{equation*}
(\ker T\pi)^\om= \ker T\J\quad\text{ and }\quad\ker T\pi\subset(\ker T\J)^\om.
\end{equation*}
Contrary to what we get in example ${\bf (i)}$ above, in infinite dimensions one can not even conclude
\eqref{reduction} is a weak dual pair, since the reversed inclusion requested in \eqref{dual_pair} may not hold.

\medskip
In the proposition below, we explore the dual pair properties for the infinite dimensional analogue of example ${\bf (ii)}$ above.

\begin{proposition}\label{zero}
Let $\J_1$ and $\J_2$ be two infinitesimally equivariant momentum maps for two symplectic Lie algebra actions of $\mathfrak{g}_1$ and $\mathfrak{g}_2$ on the symplectic manifold $(M,\om)$. Assume that $\left(\mathfrak{g}_2\right)_M\subset\ker T{\bf J}_1$ $($\ie $\J_1$
is infinitesimally $\g_2$-invariant$)$.

Then, as in the finite dimensional case, this is equivalent to $\left(\mathfrak{g}_1\right)_M\subset\ker T{\bf J}_2$
$($\ie $\J_2$ is infinitesimally $\g_1$-invariant$)$, and also to the statement $\om(\xi_M,\et_M)=0$, for all $\xi\in\g_1$ and $\et\in\g_2$. Moreover the pair of momentum maps
\begin{equation}\label{dp1}
\g_1^*\stackrel{\J_1}{\longleftarrow}(M,\om)\stackrel{\J_2}{\longrightarrow}\g_2^*
\end{equation}
is a  weak dual pair and the Lie algebra actions commute: $\left[\xi_M,\et_M\right]=0$, for all $\xi\in\g_1$ and $\et\in\g_2$.

If in addition we assume that the Lie algebra action of $\g_2$ is transitive on level sets of $\J_1$ and the Lie algebra action of $\g_1$ is transitive on level sets of $\J_2$, that is,
\begin{equation}\label{formal_dual_pair_condition}
\left(\mathfrak{g}_2\right)_M=\ker T{\bf J}_1
\quad\text{ and }\quad
\left(\mathfrak{g}_1\right)_M=\ker T{\bf J}_2,
\end{equation}
then \eqref{dp1} is a {dual pair}.
\end{proposition}
\begin{proof} We first prove the equivalence between $\left(\mathfrak{g}_2\right)_M\subset\ker T{\bf J}_1$ and $\left(\mathfrak{g}_1\right)_M\subset\ker T{\bf J}_2$. Using \eqref{kernel_of_TJ}, the first inclusion leads to $\left(((\mathfrak{g}_1)_M)^\omega\right)^\omega\subset\left((\mathfrak{g}_2)_M\right)^\omega=\ker T{\bf J}_2$. Using $V\subset (V^\omega)^\omega$, we obtain the inclusion $\left(\mathfrak{g}_1\right)_M\subset\ker T{\bf J}_2$. This also shows that another equivalent condition is $\om(\xi_M,\et_M)=0$, for all $\xi\in\g_1$ and $\et\in\g_2$. The weak dual pair conditions \eqref{dual_pair} follow easily from these observations.

The function $-\om(\xi_M,\et_M)$ is the Hamiltonian function of the commutator $[\xi_M,\et_M]$ of infinitesimal generators for $\xi\in\g_1$ and $\et\in\g_2$. Since $\om(\xi_M,\et_M)=0$, one obtains commuting Lie algebra actions.

The dual pair property is seen upon writing $(\ker
T\J_1)^\om=((\g_2)_M)^\om=\ker T\J_2$ (and similarly for $(\ker
T\J_2)^\om=\ker T\J_1$)
\end{proof}

\begin{remark}[Comparison with the finite dimensional case]\rm Note that, as in the finite dimensional case, the infinitesimal invariance conditions $\left(\mathfrak{g}_2\right)_M\subset\ker T{\bf J}_1$ and $\left(\mathfrak{g}_1\right)_M\subset\ker T{\bf J}_2$ are equivalent. However, now the transitivity assumptions $\left(\mathfrak{g}_2\right)_M=\ker T{\bf J}_1$ and $\left(\mathfrak{g}_1\right)_M=\ker T{\bf J}_2$ are not equivalent and one needs to impose both of them to have a dual pair.
\end{remark}

\begin{remark}\label{g1g2}
{\rm The necessary and sufficient condition \eqref{formal_dual_pair_condition}
for dual pairs of momentum maps can be rewritten as 
\[
\left(\g_1\right)_M^\om=\left(\g_2\right)_M\text{ and }\left(\g_2\right)_M^\om=\left(\g_1\right)_M.
\]}
\end{remark}

The preceding proposition establishes sufficient conditions for weak dual pairs and dual pairs of momentum maps in terms of infinitesimal Lie algebra actions. However, the literature provides analogous results in terms of Lie group actions. In particular, one has the following result in infinite dimensions, as consequence of Proposition \ref{zero}

\begin{corollary}\label{cipolla}
Let $\J_1$ and $\J_2$ be equivariant momentum maps arising from the canonical actions of two Lie groups $G_1$ and $G_2$ on a symplectic manifold $(M,\om)$. Assume that $\J_1$ is $G_2$-invariant $($or vice-versa$)$,
then the pair of momentum maps
\begin{equation}\label{dp2}
\g_1^*\stackrel{\J_1}{\longleftarrow}(M,\om)\stackrel{\J_2}{\longrightarrow}\g_2^*
\end{equation}
is a  weak dual pair. 
Moreover, if the $G_2$ action is transitive on level sets of $\J_1$ and vice-versa, then \eqref{dp2} is a dual pair.
\end{corollary}

An interesting case is provided by the situation below
\begin{corollary}
With the preceding hypothesis, if the groups $G_1$ and $G_2$ are connected, then their actions commute.
\end{corollary} 
\begin{proof}
Since the $G_1$-invariance of $\J_2$ implies $(\g_1)_M\subset\ker T\J_2$, the corresponding Lie algebra actions commute by Proposition \ref{zero}. Since $G_1$ and $G_2$ are connected, their actions also commute.
\end{proof}

\medskip

Also, when $M=T^*Q$ is a cotangent bundle with canonical symplectic form, we obtain the following result on commuting actions

\begin{corollary}\label{finocchio}
Given commuting actions of two Lie groups $G_1$ and $G_2$ on a manifold $Q$, and the lift of these actions
to its cotangent bundle $T^*Q$,
the pair of cotangent momentum maps
\begin{equation}\label{dp3}
\g_1^*\stackrel{\J_1}{\longleftarrow}T^*Q\stackrel{\J_2}{\longrightarrow}\g_2^*
\end{equation}
is a weak dual pair.
\end{corollary}

\begin{proof}
If $M=T^*Q$ is a cotangent bundle on which $G$ acts by the cotangent lift of its action on $Q$, then
\begin{equation}\label{cot_momap}
\mathbf{J}:T^*Q\rightarrow\mathfrak{g}^*,\quad \langle\mathbf{J}(\alpha_q),\xi\rangle=\langle\alpha_q,\xi_Q(q)\rangle
\end{equation}
is an equivariant momentum map \cite{MaRa99}. If $\xi_Q$ is an infinitesimal generator of the $G_1$ action, then $\xi_Q$ is $G_2$-equivariant, since the actions commute. Thus, the cotangent momentum map $\mathbf{J}_1$, associated to the cotangent lifted action of $G_1$, is $G_2$-invariant. The weak dual pair property follows by Corollary \ref{cipolla}.
\end{proof}

\begin{remark}[$(\pm)$ Lie-Poisson brackets]{\rm
In the above arguments, one has to consider that, if one of the
two momentum maps arises from a \textit{right} (resp. \textit{left}) action, then the
corresponding Lie-Poisson bracket carries the \textit{minus} (resp. \textit{plus}) sign.}
\end{remark}


\section{Momentum maps and central extensions}\label{trois}

In this section we present a reformulation of the ideal fluid ``dual pair" \cite{MaWe83}
that allows us to show in a rigorous way the momentum map properties of the two legs.
Given a symplectic manifold $(M,\omega)$ and a compact $k$-dimensional volume manifold $(S,\mu)$, we first show that the Fr\'echet manifold $\F(S,M)$ is in a natural way a symplectic manifold. It is clear that the subgroups $\operatorname{Diff}(M,\omega)$ and $\operatorname{Diff}(S,\mu)$  act on $\F(S,M)$ in a symplectic way. However, in order to show that these actions admit momentum mappings, 
it is necessary to restrict the actions to the groups $\operatorname{Diff}_{ham}(M,\omega)$ and $\operatorname{Diff}_{ex}(S,\mu)$ of Hamiltonian and exact volume preserving diffeomorphisms, respectively. In order to write down a momentum map in a consistent way,
one has to choose particular Hamiltonian functions on $M$ and potential forms on $S$,
but these choices introduce nonequivariance cocycles.
We solve this problem as usual by passing to the associated central extensions of the corresponding Lie algebras. This approach leads naturally to the group of quantomorphisms, a central extension of $\operatorname{Diff}_{ham}(M,\omega)$, and to the Ismagilov central extension of $\operatorname{Diff}_{ex}(S,\mu)$.

Throughout the Section we will make use of some important but standard facts regarding momentum maps and nonequivariance cocycles. For convenience of the reader and to fix our conventions, these facts are recalled in \S\ref{momap_noneq} of the Appendix.

\paragraph{The symplectic form on $\F(S,M)$.}
Let $(M,\om)$ be a symplectic manifold and $S$ a compact $k$--dimensional manifold with a fixed volume form $\mu$. The set $\mathcal{F}(S,M)$ of smooth functions from
$S$ to $M$ can be endowed with the structure of a Fr\'echet manifold in a natural way, see \cite{KrMi97}. We will denote by $U_f, V_f$ the tangent vectors to $\mathcal{F}(S,M)$ at $f$. These tangent vectors are identified with vector fields on $M$ along $f$, \ie sections of the pull-back vector bundle $f^*TM\rightarrow S$.

We now show that the $2$-form $\bar\omega$ on $\mathcal{F}(S,M)$ defined by
\[
\bar\om_f(U_f,V_f)=\int_S\om(U_f,V_f)\mu,
\]
is a symplectic form. The map $\omega\rightarrow\bar\omega$ which associates a $k$-form on $\mathcal{F}(S,M)$ to a $k$-form on $M$ will play an important role in the paper and will be referred to as the \textit{bar map}. Its properties are summarized in the Appendix of the paper. In particular, from Proposition \ref{cor2} we have $\mathbf{d}\bar\omega=\overline{\mathbf{d}\omega}=0$, so $\bar\omega$ is closed.
Note that this property cannot be deduced from a straightforward computation, since vectors $U_f$ in $T_f\mathcal{F}(S,M)$ are not necessarily of the form $U_f=u\circ f$, where $u$ is a vector field on $M$.

The (weakly) non-degeneracy of $\bar\om$ can be verified as follows.
Let $U_f$ be a non-zero vector field on $M$ along $S$, so $U_f(x)\ne 0$ for some $x\in S$. Because $\om$ is non-degenerate, one can find another vector field $V_f$ along $f$ such that $\om(U_f,V_f)$ is a bump function on $S$. Then $\bar\om(U_f,V_f)=\int_S\om(U_f,V_f)\mu\ne 0$, so $U_f$ does not belong to the kernel of $\bar\om$. This proves that the kernel of $\bar\om$ is trivial.


\subsection*{The left momentum map}

Consider the action of the group $\Diff(M,\om)=\{\varphi\in\operatorname{Diff}(M)\mid\varphi^*\omega=\omega\}$ of \textit{symplectic diffeomorphisms} of $M$ on the Fr\'echet manifold $\F(S,M)$ by composition on the left. A direct computation 
shows that this action preserves the symplectic form $\bar\om$ on $\mathcal{F}(S,M)$. In order to show the existence of a momentum map, we need to restrict this action to the subgroup
$\operatorname{Diff}_{ham}(M,\om)\subset\Diff(M,\om)$ of
\textit{Hamiltonian diffeomorphisms}. Recall that $\operatorname{Diff}_{ham}(M,\om)$ consists of symplectic diffeomorphisms which are
endpoints of Hamiltonian isotopies, i.e. isotopies defined by time dependent Hamiltonian vector fields. The Lie algebra of this
group consists of Hamiltonian vector fields
$\mathfrak{X}_{ham}(M,\om)=\{X_h\mid
h\in\mathcal{F}(M)\}$. When the first cohomology group of
$M$ vanishes, then $\operatorname{Diff}_{ham}(M,\om)$ coincides with the
identity component of $\Diff(M,\om)$. We refer to \S43.13 in
\cite{KrMi97} for further informations. The infinitesimal
generator associated to a Lie algebra element $X_h\in\mathfrak{X}_{ham}(M,\omega)$ is 
\[
\bar X_h(f)=X_h\circ f. 
\]

\begin{lemma}
The infinitesimal generator $\bar X_h$ is a Hamiltonian vector field on $\F(S,M)$ relative to the symplectic form $\bar\omega$ with the Hamiltonian function $\bar h$ given by
\[
\bar h(f):=\int_S (h\circ f)\mu,\quad f\in\F(S,M).
\]
\end{lemma}
\begin{proof} 
{From $\mathbf{i}_{X_h}\om=\dd h$ it follows that
\begin{gather*}
\mathbf{i}_{\bar X_h}\bar\om=\overline{\mathbf{i}_{X_h}\om}=\overline{\dd h}=\dd\bar h,
\end{gather*}
from the properties of the bar map in the Appendix.}
\end{proof}

\medskip

Therefore, one could define a momentum map $\mathbf{J}_L$ by the formula $\langle\mathbf{J}_L(f),X_h\rangle=\bar h(f)$.
However, the expression of $\mathbf{J}_L$ is not well-defined since the right hand side depends on a particular choice of the Hamiltonian function associated to $X_h$.

We can solve this problem by fixing one point $m_i$ in each connected component of $M$, 
and consider the unique Hamiltonian function $h_0$ of $X_h$ vanishing at these points.
The corresponding momentum map is
\begin{equation}\label{leftj}
\mathbf{J}_L:\mathcal{F}(S,M)\to \X_{ham}(M,\om)^*,\quad
\langle\mathbf{J}_L(f),X_h\rangle=\bar h_0(f)=\int_S(h_0\o f)\mu,
\end{equation}
with values in the distributional dual $\X_{ham}(M,\om)^*$.

A new problem appears here since this momentum map is not equivariant and, therefore, is not a Poisson map.
Using formula \eqref{sig} of the Appendix, we compute the $H^0(\F(S,M))$-valued Lie algebra 2--cocycle on $\X_{ham}(M,\om)$ measuring the nonequivariance of $\J_L$. For all $f\in\F(S,M)$, we have
\begin{align}\label{cocycle_large}
\si(X,Y)(f)&=-\langle\mathbf{J}_L(f),[X,Y]\rangle
-\bar\om\left(\bar X(f),\bar Y(f)\right)\nonumber\\
&=\overline{\om(X,Y)_0}(f)-\bar\om\left(\bar X,\bar Y\right)(f)
=\overline{\om(X,Y)_0-\om(X,Y)}(f),
\end{align}
where we used the properties of the bar map and the formula $[X,Y]=-X_{\omega(X,Y)}$ valid for any $X,Y\in\mathfrak{X}_{ham}(M,\omega)$.
Proposition \ref{ue} ensures that the Lie algebra action of the central extension $\widehat{\X_{ham}(M,\om)}$ of 
$\X_{ham}(M,\om)$ by $H^0(\F(S,M))$,
with characteristic cocycle $\si$, admits the infinitesimally equivariant momentum map 
\begin{equation*}
\hat\J_L:\F(S,M)\to\widehat{\X_{ham}(M,\om)}^*,\quad\hat{\mathbf{J}}_L(f)=(\mathbf{J}_L(f),-[f]),
\end{equation*}
where $[f]\in H_0(\F(S,M))$ denotes the connected component of $f$.

{We now use Proposition \ref{te} to write an infinitesimally equivariant momentum map associated to a more natural central extension of  $\X_{ham}(M,\om)$, namely, an extension by $H^0(M)$. In order to do this, we observe that the cocycle $\si$ can be written in the form $\si=T\o\si_T$,
where $\si_T$ is the $H^0(M)$-valued cocycle:
\begin{equation}\label{cocycle_T}
\si_T(X,Y)=\om(X,Y)_0-\om(X,Y),
\end{equation}
and $T:H^0(M)\to H^0(\F(S,M))$ is the linear map induced by the bar map.}
This is seen from the computation
\begin{equation*}
\langle T(\si_T(X,Y)),[f]\rangle=\overline{\si_T(X,Y)}(f)=\overline{\om(X,Y)_0-\om(X,Y)}(f)=\langle\si(X,Y),[f]\rangle.
\end{equation*}
Following the notations of Proposition \ref{te}, we denote by $\widehat{\X_{ham}(M,\omega)}_T$ the central extension of $\X_{ham}(M,\omega )$ by $H^0 (M) $ defined by $\si_T$.

{We now show that $\widehat{\X_{ham}(M,\om)}_T$ is isomorphic to a well known extension of the Lie algebra of Hamiltonian vector fields, namely, to the Lie algebra $(\F(M),\{\ ,\ \})$ of smooth functions on $M$ with canonical Poisson bracket. A section of the central extension
\begin{equation}\label{fm}
0\to H^0(M)\to\F(M)\to\X_{ham}(M,\om)\to 0
\end{equation}  
is obtained} by assigning to $X_h$ the unique Hamiltonian function $h_0$ of $X_h$ vanishing at all points $m_i$.
{The} characteristic cocycle computed with this section is exactly $\si_T$. Written in the basis $[m_i]^*$ of $H^0(M)$, 
dual to the basis $[m_i]$ of $H_0(M; \mathbb{R}  )$, the characteristic cocycle becomes 
\begin{equation}\label{cocycle_m_p}
\si_T(X,Y)=(-\om(X,Y)(m_1),\dots,-\om(X,Y)(m_p)).
\end{equation}
For connected $M$ we have $\si_T(X,Y)=-\om(X,Y)(m)$ for arbitrary $m\in M$.

The Lie algebra isomorphism between the central extension $\widehat{\X_{ham}(M,\om)}_T$
defined by $\si_T$ and $\F(M)$ is
\begin{equation}\label{LB_isom}
h\in \mathcal{F}(M)\mapsto (X_h,h_0-h)\in \widehat{\mathfrak{X}_{ham}(M,\omega)}_T,
\end{equation}
where $h_0-h\in H^0(M)=\ker(\dd:\F(M)\to\Om^1(M))$. 

We can apply Proposition \ref{te} to get an infinitesimally equivariant momentum map for the central extension
$\widehat{\X_{ham}(M,\om)}_T$ of $\X_{ham}(M,\om)$ by $H^0(M)$ defined by $\si_T$
\begin{equation}\label{jt}
\hat\J_L^T:\F(S,M)\to\widehat{\X_{ham}(M,\om)}^*_T,\quad\hat\J_L^T(f)=(\J_L(f),-T^*([f])),
\end{equation}
where $T^*:H_0(\F(S,M))\to H_0(M)$ denotes the dual of $T$.
The infinitesimally equivariant momentum map \eqref{jt},
transferred to $\F(M)$ with the help of the Lie algebra isomorphism \eqref{LB_isom},
takes the simple form:
\begin{align*}
\left\langle\hat{\mathbf{J}}_L(f),h\right\rangle
&=\left\langle\hat{\mathbf{J}}^T_L(f),(X_h,h_0-h)\right\rangle
=\left\langle{\mathbf{J}}_L(f),X_h\right\rangle-\left\langle T^*([f]),{h_0-h}\right\rangle\\
&=\int_S (h_0\circ f)\mu-\overline{(h_0-h)}(f)=\int_S(h\circ f)\mu.
\end{align*}

\medskip

When the cohomology class of the symplectic form $\omega$ is integral, i.e. $[\om]\in H^2(M,\ZZ)$, then the Lie algebra central extension
$\widehat{\X_{ham}(M,\om)}\simeq\F(M)$ can be integrated to a group central extension, namely
the group of \textit{quantomorphisms} 
as a central extension 
of the group of Hamiltonian diffeomorphisms, see \cite{Ko70} \cite{So70}. 
More precisely, in this case there exists a principal $S^1$-bundle $\pi:P\rightarrow M$ and a principal connection $\theta$ on $P$ 
with curvature $\om$, \ie $\mathbf{d}\theta=\pi^*\omega$. The group of quantomorphisms $\operatorname{Quant}(P,\theta)$ is defined as the group of connection preserving automorphisms of $P$.
This group acts on the left on $\mathcal{F}(S,M)$ by the action of $\operatorname{Diff}_{ham}(M)$, that is $\psi\cdot f=\varphi\circ f$, where $\varphi$ is the Hamiltonian diffeomorphism on the base induced by the quantomorphism $\psi$.

In general, this group extension cannot be described by a global group 2--cocycle on $\operatorname{Diff}_{ham}(M,\omega)$, i.e., the central extension is not diffeomorphic to a direct product. 
However, in the particular case when $\om$ is an exact symplectic form, such a group 2--cocycle on $\Diff_{ham}(M,\om)$ exists and is presented in \cite{IsLoMi06}.
The properties of the left momentum map obtained so far are summarized in the following theorem.

\begin{theorem} 
Let $S$ be a compact manifold endowed with a volume form $\mu$ and let $(M,\omega)$ be a symplectic manifold. Then
\begin{equation}\label{hat_momap_left}
\hat{\mathbf{J}}_L:\mathcal{F}(S,M)\rightarrow\mathcal{F}(M)^*,\quad \left\langle\hat{\mathbf{J}}_L(f),h\right\rangle
=\bar h(f)=\int_S(h\circ f)\mu
\end{equation}
is an infinitesimally equivariant momentum map for the Hamiltonian left Lie algebra action of $\mathcal{F}(M)$, shortly $\hat\J_L(f)=f_*\mu$.

When $[\omega]\in H^2(M,\mathbb{Z})$, then this Lie algebra action can be integrated to a left Hamiltonian Lie group action of $\operatorname{Quant}(P,\theta)$ on $\mathcal{F}(S,M)$, with equivariant momentum map $\hat{\mathbf{J}}_L$.
\end{theorem}
\textbf{Proof.} From the discussion above, it suffices to check the equivariance of $\hat{\mathbf{J}}_L$ under the action of the quantomorphism group. Given $\psi\in\operatorname{Quant}(P,\theta)$, covering $\varphi\in\operatorname{Diff}_{ham}(M)$, we have
\[
\left\langle\hat{\mathbf{J}}_L(\psi\cdot f),h\right\rangle=\left\langle\hat{\mathbf{J}}_L(\varphi\circ f),h\right\rangle=\int_S(h\circ \varphi\circ f)\mu=\left\langle\hat{\mathbf{J}}_L(f),h\circ \varphi \right\rangle=\left\langle\operatorname{Ad}^*_{\psi^{-1}}\hat{\mathbf{J}}_L(f),h\right\rangle,
\]
since the adjoint action of the quantomorphism group on $\mathcal{F}(M)$ reads $\operatorname{Ad}_{\psi^{-1}} h=h\circ\varphi$.

\begin{remark}\label{compsupp}{\rm
An equivariant momentum map can be written for non-compact $M$
without passing to central Lie algebra extensions, after restricting to a Lie subalgebra of $\X_{ham}(M,\om)$.
Note that to endow the various diffeomorphism groups of $M$ with a Fr\'echet Lie
group structure, one has to consider compactly supported
diffeomorphisms, if $M$ is not compact.

We restrict to the Lie subalgebra of those compactly supported 
Hamiltonian vector fields which admit compactly supported 
Hamiltonian functions, denoted by $\X_{ham}^c(M,\om)$.
This Lie algebra can be identified with the space $\F_c(M)$ of compactly 
supported functions on $M$. In this case the compactly supported Hamiltonian function is uniquely determined by the Hamiltonian vector field in $\X_{ham}^c(M,\om)$ and the Poisson bracket of two compactly supported functions is again compactly supported.
Then the left momentum map 
\[
\J_L:\F(S,M)\to\X_{ham}^c(M,\om)^*,\quad \langle\J_L(f),X_h\rangle=\int_S(h\o f)\mu
\]
is equivariant.
}
\end{remark}

\subsection*{The right momentum map} 

Consider the action of the group $\Diff(S,\mu)=\{\ps\in\operatorname{Diff}(S)\mid\ps^*\mu=\mu\}$ of \textit{volume preserving diffeomorphisms} on $\F(S,M)$ by composition on the right. A direct computation shows that this action preserves the symplectic form $\bar\om$. 
In order to show the existence of a momentum map, we need to restrict the action to the subgroup $\operatorname{Diff}_{ex}(S,\mu)$ of \textit{exact volume preserving diffeomorphisms}. Recall that the Lie algebra of this group consists of exact divergence free vector fields, that is 
$\X_{ex}(S,\mu)=\{X\in\mathfrak{X}(S)\mid \mathbf{i}_{X}\mu \;\text{is exact}\}$. We use the notation $X_\alpha$ for an exact divergence free vector field such that $\mathbf{i}_{X_\alpha}\mu =\mathbf{d}\alpha$ with $\al\in\Om^{k-2}(S)$ called {\it potential form}. Note that if the $(k-1)$-th cohomology group of $S$ vanishes (which is the same as the vanishing of the first cohomology group, since $S$ is oriented), then $\operatorname{Diff}_{ex}(S,\mu)$ coincides with the identity component of $\operatorname{Diff}(S,\mu)$. See \cite{Ba97} for further informations concerning this group. The infinitesimal generator associated to a Lie algebra element $X_\alpha\in\mathfrak{X}_{ex}(S,\mu)$ is 
$$
\widehat X_\al(f)=Tf\circ X_\al. 
$$
The next Lemma is crucial for showing the existence of a momentum map. In the proof we will make use of the \textit{hat pairing} \cite{Vi09}
\[
\Om^p(M)\x\Om^q(S)\to\Om^{p+q-k}(\F(S,M)),\quad (\omega,\alpha)\mapsto \widehat{\om\cdot\al}
\]
whose definition and properties are recalled in the Appendix.

\begin{lemma}\label{hatal} 
The infinitesimal generator $\widehat X_\al$ is a Hamiltonian vector field on $(\F(S,M),\bar\omega)$ with Hamiltonian function $-\widehat{\om\cdot\al}$ given by
\[
-\widehat{\om\cdot\al}(f)=-\int_S f^*\omega\wedge\alpha.
\]
\end{lemma}
\begin{proof} From $\mathbf{i}_{X_\al}\mu=\dd\al$ it follows that
\begin{gather*}
\dd(\widehat{\om\cdot\al})=\widehat{\dd\om\cdot\al}+\widehat{\om\cdot\dd\al}=\widehat{\om\cdot \mathbf{i}_{X_\al}\mu}=-\mathbf{i}_{\widehat X_\al}\widehat{\om\cdot\mu}=-\mathbf{i}_{\widehat X_\al}\bar\om,
\end{gather*}
from the properties of the hat pairing in the Appendix.
\end{proof}

\medskip

Note that from this result, one is tempted to define a momentum map $\mathbf{J}_R:\F(S,M)\to\X_{ex}(S,\mu)^*$ by the formula $\langle\mathbf{J}_R(f),X_\alpha\rangle=-\widehat{\omega\cdot\alpha}(f)$. However, this expression of $\mathbf{J}_R$ is not well-defined since in general the right hand side depends on a particular choice of the potential $\alpha$ associated to $X_\alpha$.

\begin{remark}[Duality pairing]\label{pairing_alpha}{\rm Below, we will identify the dual space $\X_{ex}(S,\mu)^*$ with the space of exact
$2$-forms $\mathbf{d}\Omega^1(S)$ through the $L^2$ pairing
\begin{equation}\label{L2_pairing}
\langle \gamma,X_\alpha\rangle:=\int_S\gamma\wedge\alpha.
\end{equation}
By Stokes' theorem and exactness of $\gamma$, the last integral does not depend on the choice of the potential form
$\al$ for $X_\al$.
Relative to this pairing, the coadjoint action on $\mathbf{d}\Omega^1(S)$ is given by $\operatorname{Ad}^*_\varphi\gamma=\varphi^*\gamma$.
}
\end{remark}

\paragraph{The case of an exact symplectic form.}
If we assume that the symplectic form $\omega$
is exact, then it is readily checked that the Hamiltonian function $-\widehat{\om\cdot\al}$
is independent of the choice of the potential $\al$. In this case, the momentum map 
\begin{gather*}
\mathbf{J}_R:\mathcal{F}(S,M)\to \X_{ex}(S,\mu)^*,\quad
\langle\mathbf{J}_R(f),X_\al\rangle
=-\widehat{(\om\cdot\al)}(f)
=-\int_Sf^*\om\wedge\al
\end{gather*}
is well-defined. Using the pairing \eqref{L2_pairing}, it can be written as
\begin{equation}\label{exactj}
\mathbf{J}_R:\F(S,M)\rightarrow\mathfrak{X}_{ex}(S,\mu)^*\simeq\dd\Omega^1(S),\;\;\mathbf{J}_R(f)=-f^*\om.
\end{equation}
It is clearly equivariant since $\mathbf{J}_R(f\circ\varphi)=\varphi^*f^*\omega=\operatorname{Ad}^*_\varphi\mathbf{J}_R(f)$.

\paragraph{The case of a general symplectic form.} Recall that in order to remove the dependence of the left momentum map $\mathbf{J}_L$ on a particular potential $h$, we chose one point $m_i$ in each connected component of $M$, and considered the unique Hamiltonian function $h_0$ vanishing at these points. Equivalently, we can say that we have considered points $m_i$ in $M$ such that $[m_i]$ form a basis of $H_0(M; \mathbb{R}  )$. We shall use a similar idea to remove the dependence of $\mathbf{J}_R$ on the potential $\alpha$. Assume that $ k\geq 3$ (see below for the case $k=2$) and fix a set of $(k-2)$-dimensional submanifolds $N_1,...,N_p$ of $S$ which determine a basis of the singular homology group $H_{k-2}(S; \mathbb{R} )$. Then, given $X_\alpha\in\mathfrak{X}_{ex}(S,\mu)$, there is a unique potential $\alpha_0$ of $X_\alpha$, up to an exact form, such that
\[
\int_{N_i}\alpha_0=0,\quad\text{for all}\quad i=1,...,p.
\]
Indeed, if $\be_0$ verifies the same properties, then $\be_0-\alpha_0$ is closed and the integral $\int_{N_i}(\be_0-\alpha_0)$ vanishes for all $i=1,...,p$, which proves that $\be_0-\alpha_0$ is exact, by de Rham's Theorem. 
We conclude that the momentum map given by
\begin{equation}\label{jrzero}
\left\langle\mathbf{J}_R(f),X_\alpha\right\rangle
=-\widehat{(\om\cdot\al_0)}(f)=-\int_S f^*\omega\wedge \alpha_0,
\end{equation}
is well-defined. In particular it depends only on the equivalence class $[\alpha_0]\in\Omega^{k-2}(S)/\mathbf{d}\Omega^{k-3}(S)$ since $\omega$ is closed.


In general this momentum map is not equivariant, therefore it is not a Poisson map.
Using formula \eqref{sig} for right actionns, we now compute the $H^0(\F(S,M))$-valued Lie algebra 2--cocycle $\si$ 
on $\X_{ex}(S,\mu)$ measuring the nonequivariance of $\J_R$. We have for all $f\in\F(S,M)$
\begin{align}\label{2_cocycle}
\si(X,Y)(f)&=\langle\J_R(f),[X,Y]\rangle
-\bar\om(\hat X,\hat Y)(f)
=-\int_S(f^*\omega)\wedge\left(\ii_X\ii_Y\mu\right)_0-\int_S(f^*\omega)(X,Y)\mu\nonumber\\
&=\int_Sf^*\om\wedge\left(\ii_{X}\ii_{Y}\mu-(\ii_{X}\ii_{Y}\mu)_0\right)
=\left\langle[f^*\omega],[\ii_{X}\ii_{Y}\mu-(\ii_{X}\ii_{Y}\mu)_0]\right\rangle,
\end{align}
where we used that $\ii_X\ii_Y\mu$ is a potential form for 
the exact divergence free vector field $[X,Y]$ and the pairing between $H^2(S)$ and $H^{k-2}(S)$
in the last term.

Proposition \ref{ue} ensures that the Lie algebra action of the central extension $\widehat{\X_{ex}(S,\mu)}$ of 
$\X_{ex}(S,\mu)$ by $H^0(\F(S,M))$,
with characteristic cocycle $\si$, admits the infinitesimally equivariant momentum map 
\begin{equation*}
\hat\J_R:\F(S,M)\to\widehat{\X_{ex}(S,\mu)}^*,\quad\hat{\mathbf{J}}_R(f)=(\mathbf{J}_R(f),-[f]),
\end{equation*}
where $[f]\in H_0(\F(S,M))$ denotes the connected component of $f$.

{We now use Proposition \ref{te} to write an infinitesimally equivariant momentum map associated to a more natural central extension of  $\X_{ex}(S,\mu )$, namely, an extension by $H^{k-2}(S)$. In order to do this, we consider the linear map
\[
T:H^{k-2}(S)\to H^0(\F(S,M)),\quad T([\al])=\widehat{\om\cdot\al}
\]
induced by the hat pairing with the closed 2-form $\om$ on $M$.
Note that $T$ is well defined since the hat pairing depends only on the cohomology class of the closed form $\al$, and the differential of the function $\widehat{\om\cdot\al}$
vanishes by Proposition \ref{dddd}.}
With the identification of the dual of $H^{k-2}(S)$ with $H^2(S)$,
the dual map of $T$ can be written as
\[
T^*:H_0(\F(S,M))\to H^2(S),\quad T^*([f])=[f^*\om],
\]
because for any closed $(k-2)$-form $\al$ on $S$,
$\langle[f^*\om],[\al]\rangle=\int_S f^*\om\wedge\al=\widehat{\om\cdot\al}(f)=\langle [f],T([\al])\rangle
=\langle T^*([f]),[\al]\rangle$. We now observe that the cocycle $\si$ can be written in the form $\si=T\o\si_T$, where $\si_T$ is the
$ H^{k-2}(S)$-valued cocycle
\begin{equation}\label{cocycle_T_R}
\si_T(X,Y)=[\ii_X\ii_Y\mu-(\ii_X\ii_Y\mu)_0].
\end{equation}
This is seen from the computation
\begin{equation*}
\langle T(\si_T(X,Y)),[f]\rangle=\widehat{\om\cdot(\ii_X\ii_Y\mu-(\ii_X\ii_Y\mu)_0)}(f)=\langle\si(X,Y),[f]\rangle.
\end{equation*}
{Following the notation of Proposition \ref{te}, we denote by $\widehat{\X_{ex}(S,\mu)}_T$ the central extension of $\X_{ex}(S,\mu)$ by $H^{k-2}(S)$ defined by $\si_T$.}

{We now show that $\widehat{\X_{ex}(S,\mu)}_T$ is isomorphic to a more familiar central extension of the Lie algebra of 
exact divergence free vector fields. }
Consider the central extension
\begin{equation}\label{extension_ex}
0\to H^{k-2}(S)\to\Om^{k-2}(S)/\dd\Om^{k-3}(S)\to\X_{ex}(S,\mu)\to 0,
\end{equation}
due to \cite{Li74} \cite{Ro95}. Here the Lie algebra bracket on $\Om^{k-2}(S)/\dd\Om^{k-3}(S)$
is given by $\{[\al],[\be]\}=[\ii_{X_\be}\ii_{X_\al}\mu]$. 
One observes that the map $[\al]\mapsto X_\al$ is a surjective Lie algebra homomorphism from $\Om^{k-2}(S)/\dd\Om^{k-3}(S)$
to the Lie algebra $\X_{ex}(S,\mu)$ with opposite bracket, and 
that its kernel is given by $H^{k-2}(S)$ and is central. 

A section of \eqref{extension_ex} is obtained by assigning to $X_\al$ 
the unique potential form $\al_0$ satisfying $\int_{N_i}\alpha_0=0$ for all $i$.
The characteristic cocycle computed with this section is exactly $\si_T$ form \eqref{cocycle_T_R}. 
The Lie algebra isomorphism between the central extension $\widehat{\X_{ex}(S,\mu)}_T$
defined by $\si_T$ and $\Om^{k-2}(S)/\dd\Om^{k-3}(S)$ is
\begin{equation}\label{LR_isom}
[\al]\in \Om^{k-2}(S)/\dd\Om^{k-3}(S)\mapsto (X_\al,[\al-\al_0])\in 
\widehat{\X_{ex}(S,\mu)}_T=\X_{ex}(S,\mu)\x H^{k-2}(S).
\end{equation}

We can apply Proposition \ref{te} to get an infinitesimally equivariant momentum map for the central extension
$\widehat{\X_{ex}(S,\mu)}_T$ of $\X_{ex}(S,\mu)$ by $H^{k-2}(S)$ defined by $\si_T$
\begin{equation*}
\hat\J^T_R:\F(S,M)\to\widehat{\X_{ex}(S,\mu)}^*_T,\quad\hat\J^T_R(f)=(\J_R(f),-T^*([f]))=(\J_R(f),-[f^*\om]),
\end{equation*}
This infinitesimally equivariant momentum map, transferred to the central extension $\Om^{k-2}(S)/\dd\Om^{k-3}(S)$ with the help of the Lie algebra isomorphism \eqref{LR_isom}, is:
\begin{align}\label{jh}
\left\langle\hat{\mathbf{J}}_R(f),[\al]\right\rangle
&=\left\langle\hat{\mathbf{J}}^T_R(f),(X_\al,[\al-\al_0])\right\rangle
=\left\langle{\mathbf{J}}_R(f),X_\al\right\rangle-\left\langle T^*([f]),[\al-\al_0]\right\rangle\nonumber\\
&=-\int_S f^*\om\wedge\al_0-\widehat{\om\cdot(\al-\al_0)}(f)=-\int_Sf^*\om\wedge\al.
\end{align}
Since the (regular) dual of $\Om^{k-2}(S)/\dd\Om^{k-3}(S)$ can be identified
with the space of closed $2$-forms on $S$
through the $L^2$ pairing 
\begin{equation}\label{etalpha}
\langle\et,[\al]\rangle=\int_S\et\wedge\al,
\end{equation}
the equivariant momentum map \eqref{jh} becomes 
\[
\hat\J_R:\F(S,M)\to (\Om^{k-2}(S)/ \mathbf{d} \Om^{k-3}(S))^*=Z^2(S),\quad\hat\J_R(f)=-f^*\om.
\]

The central extension \eqref{extension_ex} is universal \cite{Ro95}, in  particular 
$H^2(\X_{ex}(S,\mu))\simeq H^2(S)$,
the isomorphism being given by the map which associates to a cohomology class $[\et]\in H^2(S)$ the cohomology class $[\si_\et]$ of the Lichnerowicz cocycle
\[
\sigma_\eta(X,Y)=\int_S\eta(X,Y)\mu,\quad X,Y\in\mathfrak{X}(S,\mu).
\]

Given a connected component $\F_0$ of $\F(S,M)$, we compute the nonequivariance of 
$\J_R$ on $\F_0$:
\begin{align}\label{LA_2_cocycle}
\si_{\F_0}(X,Y)&=\si(X,Y)(f)=T(\si_T(X,Y))(f)=\langle[f^*\om],\si_T(X,Y)\rangle, f\in\F_0.
\end{align}
As it should, the right hand side of \eqref{LA_2_cocycle} does not depend on $f\in\F_0$ since the cohomology class $[f^*\omega]$ does not depend on $f\in\F_0$.
The cocycle \eqref{LA_2_cocycle} is cohomologous to the Lichnerowicz cocycle $\si_\et$ with $\et=-f^*\om$.
Indeed, because the assignement $\al\mapsto \al_0$ depends only on $X_\al$, hence on $\dd\al$,
we can write $\al_0=b(\dd\al)$ for some linear map $b:\dd\Om^{k-2}(S)\to\Om^{k-2}(S)$ (a right inverse to $\dd$). 
Using $\dd\ii_X\ii_Y\mu=\ii_{[X,Y]}\mu$, we compute the difference
\begin{align*}
(\si_\et-\si_{\F_0})(X,Y)
&=-\int_M({f^*\om})(X,Y)\mu-\int_M f^*\om\wedge(\ii_X\ii_Y\mu-(\ii_X\ii_Y\mu)_0)\\
&=\int_M f^*\om\wedge (\ii_X\ii_Y\mu)_0
=\int_Mf^*\om\wedge b\dd\ii_X\ii_Y\mu=\int_M f^*\om\wedge b(\ii_{[X,Y]}\mu)
\end{align*}
and we notice it is a coboundary.
For cocycles cohomologous to Lichnerowicz cocycles see \cite{Vi10}.

\medskip

Assume without loss of generality that $\int_S\mu=1$. The central extended Lie algebra $\Om^{k-2}(S)/\dd\Om^{k-3}(S)$ admits an underlying Lie group central extension of $\operatorname{Diff}_{ex}(S,\mu)$ by the torus $H^{k-2}(S)/L$,
where 
\[
L=\left\{[\al]\in H^{k-2}(S)\ :\ \int_{N_i}\al\in\ZZ, \;i=1,\dots,p\right\}
\]
with $N_1,\dots,N_p$ fixed $(k-2)$-dimensional submanifolds of $S$ which determine a basis for $H_{k-2}(S,\RR)$
as above, see \cite{Is96} Section 25.5 for details.

\begin{theorem}\label{equiv_right_momap}
Let $S$ be a compact manifold endowed with a volume form $\mu$, and let $(M,\omega)$ be symplectic manifold. Then
\begin{equation}\label{hat_momap_right}
\hat{\mathbf{J}}_R:\mathcal{F}(S,M)\rightarrow
(\Om^{k-2}(S)/\dd\Om^{k-3}(S))^*,\quad \left\langle\hat{\mathbf{J}}_R(f),[\al]\right\rangle=-\widehat{\om\cdot\al}(f)=-\int_Sf^*\om\wedge\al
\end{equation}
is an infinitesimally equivariant momentum map for the Hamiltonian right Lie algebra action of the central extension $\Om^{k-2}(S)/\dd\Om^{k-3}(S)$ of $\X_{ex}(S,\mu)$.
With the identification \eqref{etalpha} the momentum map can be written shortly $\hat\J_R(f)=-f^*\om\in Z^2(S)$.

Assuming  that $\int_S\mu=1$, this Lie algebra action can be integrated to a right Hamiltonian Lie group action on $\mathcal{F}(S,M)$, of
the central extension of the group $\Diff_{ex}(S,\mu)$ of exact
volume preserving diffeomorphisms by the torus $H^{k-2}(S)/L$, 
with equivariant momentum map $\hat{\mathbf{J}}_R$.
\end{theorem}

\begin{proof} From the discussion above, it suffices to check the equivariance of $\hat\J_R$.
Given an element $\Ps$ in the central extension of the group $\Diff_{ex}(S,\mu)$,
sitting over $\ps\in\Diff_{ex}(S,\mu)$, the right action of $\Psi$ on $f\in\F(S,M)$ is $\Psi\cdot f=f\o\ps$, and
\[
\langle\hat\J_R(\Ps\cdot f),[\al]\rangle
=-\int_S(f\o\ps)^*\om\wedge\al
=-\int_Sf^*\om\wedge(\ps^{-1})^*\al
=\langle\Ad^*_{\Ps}\hat\J_R(f),[\al]\rangle
\]
since the adjoint action in the central extended group is $\Ad_{\Ps}[\al]=[(\ps^{-1})^*\al]$
for all $[\al]\in \Om^{k-2}(S)/\dd\Om^{k-3}(S)$.
\end{proof}
All the considerations above, including Theorem \ref{equiv_right_momap}, are still valid for $k=2$, but in this case the central extension \eqref{extension_ex} becomes
\[
0\to H^0(S)\to\F(S)\to\X_{ham}(S,\mu)\to 0
\]
and is trivial because $S$ is compact. To see this it suffices to make use of the Lie algebra isomorphism $\alpha \in \mathcal{F} (S) \mapsto \left(  X_ \alpha ,\left(\int_{S_1} \alpha \mu,\dots,\int_{S_p} \alpha \mu\right) \right) \in \mathfrak{X}_{ham}(S)\times H^0(S)$, where $S_1,\dots,S_p$ are the connected components of $S$. This suggests that we can find an equivariant right momentum map without passing to central extensions of $\X_{ham}(S,\mu)$. Indeed, to each Hamiltonian vector field $X_\al$ {(or to each function $ \alpha $)} we can associate the unique Hamiltonian function $\al^0\in\F(S)$ {of $X_ \alpha $} with vanishing integral $\int_{S_i}\al^0\mu$ over each connected component $S_i$ of $S$. Then, using the function $\al^0$ instead of $\al_0$ in \eqref{jrzero}, we get an equivariant momentum map
$\langle\J_R(f),X_\al\rangle=-\int_Sf^*\om\wedge\al^0$. To see this, we notice that
$\mu(X,Y)\mu$ is an exact 2-form for all Hamiltonian vector fields $X,Y\in\X_{ham}(S,\mu)$,
so $\int_{S_i}\mu(X,Y)\mu=0$ for all $i$. This implies $\mu(X,Y)^0=\mu(X,Y)$, and we compute for all $f\in \mathcal{F} (S,M)$:
\begin{align*}
-\langle\J_R(f),[X,Y]\rangle&=\int_Sf^*\om\wedge\mu(X,Y)^0
=\int_Sf^*\om\wedge\mu(X,Y)=\int_S(f^*\om)(X,Y)\mu\\
&=\bar\om(\hat X,\hat Y)(f).
\end{align*}


\section{Dual pair for Euler equation}\label{quatre}

In this section, we show that the pair of momentum maps $\hat\J_L$ and $\hat\J_R$ obtained in \S\ref{trois} form a weak dual pair on the symplectic manifold $(\F(S,M),\bar\omega)$. We then prove that the restriction of this weak dual pair to
the open subset $\operatorname{Emb}(S,M)$ is a dual pair under the topological condition $H^1(S)=0$.

\medskip

The natural actions of the Hamiltonian group $\Diff_{ham}(M,\om)$ and of the group $\Diff_{ex}(S,\mu)$ of exact volume preserving diffeomorphisms are two commuting symplectic actions on the symplectic manifold $\F(S,M)$. As we have shown in \S\ref{trois}, the induced infinitesimal actions of the central extensions $\widehat{\X_{ham}(M,\om)}=\F(M)$ and $\widehat{\X_{ex}(S,\mu)}=\Om^{k-2}(S)/\dd\Om^{k-3}(S,\mu)$ admit the infinitesimally equivariant momentum maps $\hat\J_L(f)=f_*\mu$ and 
$\hat\J_R(f)=-f^*\om$.

\begin{lemma} The pair of momentum maps

\begin{picture}(150,100)(-70,0)%
\put(95,75){$(\F(S,M),\bar\om)$}

\put(90,50){$\hat{\J}_L$}

\put(160,50){$\hat{\J}_R$}

\put(2,15){$\Den_c(M)=\widehat{\X_{ham}(M,\om)}^*$}

\put(150,15){$\widehat{\X_{ex}(S,\mu)}^*=Z^2(S)$}

\put(130,70){\vector(-1, -1){40}}

\put(135,70){\vector(1,-1){40}}

\end{picture}\\
is a weak dual pair.
\end{lemma}
\begin{proof} This follows from Proposition \ref{zero}. One has only to check that the symplectic form
$\bar\om$ vanishes on pairs of infinitesimal generators for the two
commuting actions: 
\begin{align}\label{ooo}
\bar\om\left(h_{\mathcal{F}(S,M)},[\al]_{\mathcal{F}(S,M)}\right)
&=\bar\om(\bar X_h,\hat X_\al)=\ii_{\hat X_\al}\overline{\dd h}
=\ii_{X_\al}\widehat{\dd h\cdot\mu}=\widehat{\dd h\cdot\dd\al}=0.
\end{align}
We conclude that $\hat\J_L$ and $\hat\J_R$ determine a weak dual pair structure. 
\end{proof}

\medskip

For the proof of the dual pair property \eqref{strong_dual_pair} we need a few technical lemmas. The first lemma is detached from the proof of Proposition 3 in \cite{HaVi2004}.

\begin{lemma}\label{function}
Let $(M,g)$ be a Riemannian manifold and $N\subset M$ a submanifold. Let $E$ be the normal bundle $TN^\perp$ 
viewed as a tubular neighborhood of $N$ in $M$,
with the identification done by the
Riemannian exponential map. 
Then any section $\al\in\Ga(T^*M|_N)$ vanishing on $TN$,
when restricted to $TN^\perp$, 
defines a smooth function $h\in\F(E)$,
linear on each fiber,
whose differential along $N$ satisfies $(\dd h)|_N=\al$
(as sections of $T^*M|_{N}$).
\end{lemma}

\begin{proof}
The zero section $s:N\to TN^\perp$ and the inclusion $i_x:T_xN^\perp\to TN^\perp$, $x\in N$, are used to identify $T_xM=T_xN\oplus T_xN^\perp$ with $T_{0_x}(TN^\perp)$, namely $(X_x,Y_x)\mapsto 
T_xs(X_x)+T_{0_x}i_x(Y_x)$.
Under this identification, 
\[
\dd_xh(X_x)=\dd_{0_x}\al(T_x s(X_x))=\dd_x(\al\o s)(X_x)=0=\al(X_x)
\]
for all $X_x\in T_xN$ and
\[
\dd_xh(Y_x)=\dd_{0_x}\al(T_{0_x}i_x(Y_x))=\dd_{0_x}(\al\o i_x)(Y_x)=\al(Y_x)
\]
for all $Y_x\in T_xN^\perp$, which imply the requested identity 
$(\dd h)|_N=\al$.
\end{proof}

\medskip

In particular any $1$-form $\al$ on $M$ with the property $i^*\al=0$,
where $i:N\to M$ denotes the inclusion, satisfies the condition of Lemma \ref{function} when restricted to $TM|_N$. 

\begin{lemma}\label{lem1}
Let $S$ be a compact manifold with $H^1(S)=0$ and $f:S\to M$ an embedding.
If  a vector field $Y\in\X(M)$ satisfies $f^*\pounds_Y\om=0$,
then there exists a Hamiltonian vector field $\tilde Y\in\X(M)$ such that $\tilde Y\o f= Y\o f$.
\end{lemma}

\begin{proof}
Because $Y\in\X(M)$ satisfies
$f^*\pounds_Y\om=0$, the $1$-form $f^*\mathbf{i}_Y\omega$ on $S$ is
closed. It is also exact by $H^1(S)=0$, so there exists
$h_1\in\F(S)$ with $f^*\mathbf{i}_Y\om=\mathbf{d} h_1$. Let $\tilde
h_1\in\F(M)$ be a smooth function extending $h_1$, in the sense that $h_1=\tilde h_1\o f$, and let $X_{\tilde h_1}$ be the associated Hamiltonian
vector field on $M$. Then $Z=Y-X_{\tilde h_1}$ satisfies
$f^*\mathbf{i}_Z\om=0$, so the $1$-form $\al=\mathbf{i}_Z\om$ on $M$
vanishes on vectors tangent to the submanifold $N=f(S)$ of $M$.

We apply Lemma \ref{function} to $\al$.
We consider a Riemannian metric on $M$ and we denote by
$E$ the normal bundle over the submanifold $N$ of
$M$, viewed as a tubular
neighborhood of $N$, with the identification done by the
Riemannian exponential map. 
The $1$-form $\al$ restricted to normal
vectors defines a function $h_2$ on $E$, linear on each fiber, and its
differential along $N$ is the restriction of the $1$-form $\al$,
\ie $(\dd h_2)|_{N}=\al|_{N}$ (as sections of $T^*M|_{N}$).

Let $\tilde h_2$ be a smooth function on $M$ with $\tilde h_2=h_2$ on a neighborhood of $N$.
Then the Hamiltonian vector field $\tilde Y=X_{\tilde h_1+\tilde h_2}$ has the required property
$\tilde Y\o f= Y\o f$,
because $(\mathbf{i}_{\tilde Y}\om)|_N=(\dd\tilde h_1)|_N
+(\dd\tilde h_2)|_N=(\mathbf{i}_{Y-Z}\om)|_N+\al|_N=(\mathbf{i}_Y\om)|_N$.
\end{proof}

\begin{lemma}\label{lem2}
Let $f:S\to M$ be an embedding and $Y\in\X(M)$.
If $\int_S f^*(\pounds_Yh)\mu =0$ for all functions $h\in\F(M)$,
then the vector field $Y$ is tangent to the
submanifold $f(S)$ of $M$, \ie there exists a vector field $Z\in\X(S)$ with $Y\o f=Tf\o Z$,
which means that the vector fields $Y$ and $Z$ are $f$-related.
\end{lemma}

\begin{proof}
Since $f$ is an embedding, without loss of generality we can assume that $f$ is the inclusion of a submanifold $S$ of $M$. We have to show that if $\int_S (\pounds_Yh)\mu =0$ for all functions $h\in\F(M)$, then $Z=Y|_S\in\X(S)$.

We assume by contradiction that there exists $s\in S$ with
$Y(s)$ not tangent to $S$. We consider a coordinate chart
$u:U\to(-1,1)^n$ on $M$ centered at $s$ with coordinates
$(x_1,\dots,x_k,y_{k+1},\dots,y_n)$, such that points on $S\cap U$
have coordinates $(x_1,\dots,x_k,0,\dots,0)$ and such that locally
$Y|_U\simeq\partial_{y_n}$. With the help of the function
\[
h_0:U\to\RR,\quad h_0(x_1,\dots,x_k,y_{k+1},\dots,y_n)= y_n
b_1(x_1,\dots,x_k)b_2(y_{k+1},\dots, y_{n}),
\]
where $b_1:(-1,1)^k\to \mathbb R$ and $b_2:(-1,1)^{n-k}\to\mathbb R$
are two bump functions identically 1 around zero, we build a smooth
function $h$ on $M$ which is zero outside $U$. The computation
$\partial_{y_n}h_0(x,y)=b_1(x)b_2(y)+y_nb_1(x)\partial_{y_n}b_2(y)$
shows that the restriction of $\pounds_Yh$ to $S$ has
support in $S\cap U$ and in this chart it is the bump function
$b_1$. It follows that
\[
\int_S(\pounds_Yh)\mu=\int_{(-1,1)^k}b_1(x_1,\dots,x_k){\rm d}x_1\dots{\rm
d}x_k\ne 0,
\]
a contradiction. We conclude that $Y|_S\in\X(S)$.
\end{proof}

\medskip

Using these two lemmas, we can now state the main result of this section.

\begin{theorem}\label{mawe}
The restriction of the weak dual pair for Euler equation 
to the open subset of embeddings is a dual pair under
the assumption $H^1(S)=0$.
\end{theorem}

\begin{proof}
First we notice that the equalities
$(\X_{ham}(M,\om))_{\Emb(S,M)}=(\widehat{\X_{ham}(M,\om)})_{\Emb(S,M)}$
and $(\X_{ex}(S,\mu))_{\Emb(S,M)}=
(\widehat{\X_{ex}(S,\mu)})_{\Emb(S,M)}$ hold.
By remark \ref{g1g2}, the conditions
\[
(\X_{ham}(M,\om))_{\Emb(S,M)}=
\left((\X_{ex}(S,\mu))_{\Emb(S,M)}\right)^{\bar\om}
\]
and
\[
(\X_{ex}(S,\mu))_{\Emb(S,M)}=
\left((\X_{ham}(M,\om))_{\Emb(S,M)}\right)^{\bar\om}
\]
guarantee that the weak dual pairs are dual pairs. The
inclusions from left to right are the weak dual pair conditions verified in \eqref{ooo}.
For the other inclusions we need the topological restriction on $S$.

Recall that the action of the diffeomorphism group $\Diff(M)$ on
$\Emb(S,M)$ is infinitesimally transitive \cite{H76}, so the tangent space to
the manifold of embeddings is $T_f\Emb(S,M)=\{\bar Y(f)=Y\o f:Y\in\X(M)\}$.

We start with an arbitrary $\bar Y(f)=Y\o f\in
((\X_{ex}(S,\mu))_{\Emb(S,M)}(f))^{\bar\om}$ with $Y\in\X(M)$. 
This means that for
all $\al\in\Om^{k-2}(S)$,
\begin{align*}
0&=\bar\om \left( \bar Y(f), \hat X_\al(f) \right) 
=\left( \ii_{\hat X_\al}\ii_{\bar Y}\widehat{\om\cdot\mu} \right) (f)
=\widehat{ \left( \ii_Y\om\cdot\ii_{X_\al}\mu \right) }(f)
=\widehat{(\ii_Y\om\cdot\dd\al)}(f)\\
&=\widehat{(\pounds_Y\om\cdot\al)}(f)
=\int_S(f^*\pounds_Y\om)\wedge\al.
\end{align*}
We used here Propositions \ref{pppp} and \ref{five} from the Appendix.
It follows that $f^*\pounds_Y\om=0$,
so by Lemma \ref{lem1} there exists a Hamiltonian vector field $\tilde Y$ on $M$ such that $\tilde Y\o f= Y\o f$.
We can conclude that $Y\o f\in(\X_{ham}(M,\om))_{\Emb(S,M)}(f)$, which shows the inclusion 
\begin{equation*}
\left((\X_{ex}(S,\mu))_{\Emb(S,M)}\right)^{\bar\om}\subset
(\X_{ham}(M,\om))_{\Emb(S,M)}.
\end{equation*}

For the inclusion
\[
\left((\X_{ham}(M,\om))_{\Emb(S,M)}\right)^{\bar\om}\subset
(\X_{ex}(S,\mu))_{\Emb(S,M)}
\]
we start with an arbitrary $\bar Y(f)=Y\o f\in
\left((\X_{ham}(M,\om))_{\Emb(S,M)}(f)\right)^{\bar\om}$, which means
that for all $h\in\F(M)$,
\begin{equation*}
0=\bar\om(\bar Y(f),\bar X_h(f))
=\overline{\om(Y,X_h)}(f)
=-\overline{\pounds_Yh}(f)
=-\int_S f^*(\pounds_Yh)\mu.
\end{equation*}
>From Lemma \ref{lem2} follows that there exists a vector field $Z$ on $S$ with $\bar Y(f)=Y\o f=Tf\o Z=\hat Z(f)$ for $Z\in\X(S)$. 

Now, since $\hat Z(f)\in \left((\X_{ham}(M,\om))_{\Emb(S,M)}(f)\right)^{\bar\om}$, 
for every $h\in\mathcal{F}(M)$ we have
\begin{align*}
0&=\bar\om \left( \hat Z(f),\bar X_h(f) \right) 
=(\ii_{\bar X_h}\ii_{\hat Z}\widehat{\om\cdot\mu})(f)
=-\widehat{(\dd h\cdot\ii_Z\mu)}(f)
=\widehat{h\cdot\pounds_Z\mu}(f)
=\int_S(f^*h)\pounds_Z\mu,
\end{align*}
hence $\pounds_Z \mu =0$ ($f$ is an embedding, so all smooth
functions on $S$ can be written in the form $h\o f$ with
$h\in\F(M)$). This means that $Z$ is a divergence free vector field
on $S$. 
Since the manifold $S$ is oriented, $H^1(S)=0$ implies
$H^{k-1}(S)=0$, so $Z$ is an exact divergence free vector field. We conclude that $\bar Y(f)\in
(\X_{ex}(S,\mu))_{\Emb(S,M)}$.
\end{proof}

\begin{remark}[The case of an exact symplectic form]\label{45}
{\rm There are two variants of the Marsden-Weinstein dual pair in the special case of an exact symplectic 
(hence non-compact) manifold $M$. As we have seen in \S\ref{trois}, in this special case the central extension of $\mathfrak{X}_{ex}(S,\mu)$ is not needed and we have $\J_R(f)=-f^*\om\in\mathfrak{X}_{ex}(S,\mu)^*\simeq\mathbf{d}\Omega^1(S)$
from \eqref{exactj}.
For the left leg we consider as above the left action of the Lie algebra of smooth functions on $M$, as central extension of the Lie algebra of Hamiltonian vector fields, so we get $\hat{\J}_L(f)=f_*\mu$ and the diagram

\begin{picture}(150,100)(-70,0)%
\put(95,75){$(\Emb(S,M),\bar\om)$}

\put(90,50){$\hat{\J}_L$}

\put(160,50){${\J}_R$}

\put(8,15){$\Den_c(M)={\F(M)}^*$}

\put(150,15){${\X_{ex}(S,\mu)}^*=\dd\Om^1(S)$.}

\put(130,70){\vector(-1, -1){40}}

\put(135,70){\vector(1,-1){40}}

\end{picture}\\
The dual pair property follows from Theorem \ref{mawe}.

Another possibility for the left leg is to consider, like in Remark \ref{compsupp}, the left action of the Lie algebra of those compactly supported 
Hamiltonian vector fields on $M$ who admit compactly supported 
Hamiltonian functions. This Lie algebra, denoted by $\X_{ham}^c(M,\om)$, can be identified with the space $\F_c(M)$ of compactly supported functions on $M$, so we have $\J_L(f)=f_*\mu$ and a weak dual pair

\begin{picture}(150,100)(-70,0)%
\put(95,75){$(\Emb(S,M),\bar\om)$}

\put(90,50){${\J}_L$}

\put(160,50){${\J}_R$}

\put(0,15){$\Den(M)={\X_{ham}^c(M,\om)}^*$}

\put(150,15){${\X_{ex}(S,\mu)}^*=\dd\Om^1(S)$.}

\put(130,70){\vector(-1, -1){40}}

\put(135,70){\vector(1,-1){40}}

\end{picture}\\
To show it is a dual pair, one uses $(\X_{ham}(M,\om))_{\Emb(S,M)} 
=(\X_{ham}^c(M,\om))_{\Emb(S,M)}$ in the proof of Theorem \ref{mawe}.
}\end{remark}


\section{Dual pair for EPDiff equation}\label{cinq}

In this Section, we show that the pair of momentum maps associated to the EPDiff equation (\cite{HoMa2004}) form a weak dual pair on $T^*\F(S,M)$ and a dual pair on the subset $T^*\operatorname{Emb}(S,M)^\times$ of $T^*\operatorname{Emb}(S,M)$ consisting of nowhere vanishing one-form densities along embeddings.

The left action of $\operatorname{Diff}(M)$ and the right action of $\Diff(S)$ on $\F(S,M)$ lift to symplectic actions on the cotangent bundle $T^*\F(S,M)$ endowed with the canonical symplectic form $\Om$.
We now compute the cotangent momentum maps associated to these actions. Let $X\in\X(M)$, $Y\in\X(S)$, and $P\in T^*\F(S,M)$. Note that $P$ is a 1-form density on $M$ along the mapping $Q\in\F(S,M)$ given by $Q=\pi_{\F(S,M)}(P)$, where $\pi_{\F(S,M)}:T^*\F(S,M)\to\F(S,M)$ denotes the canonical projection.

Applying the general formula \eqref{cot_momap} for the cotangent momentum map, we compute
\[
\left\langle\J_L(P), X \right\rangle=\left\langle P,\bar X(Q)\right\rangle
=\int_S P\!\cdot\!( X \o Q)=\left\langle\int_S P\de(x-Q), X \right\rangle
\]
and
\[
\left\langle\J_R(P),Y \right\rangle=\left\langle P,\hat Y(Q)\right\rangle
=\int_S P\!\cdot\!(TQ\o  Y )=\langle P\!\cdot\! TQ, Y \rangle.
\]
The momentum maps $\J_L$ and $\J_R$ are thus given by
\begin{equation}\label{trei}
\J_L:T^*\F(S,M)\to\X(M)^*,\quad\mathbf{J}_L(P)=
\int_S P\delta(x- Q),
\end{equation}
and
\begin{equation}\label{patru}
\J_R:T^*\F(S,M)\to\X(S)^*,\quad\mathbf{J}_R(P)= P \!\cdot\!T Q .
\end{equation}
They form the pair of momentum maps  for EPDiff equation considered by \cite{HoMa2004}:

\begin{picture}(150,100)(-70,0)%
\put(85,75){$(T^*\F(S,M),\Om)$}

\put(90,50){$\mathbf{J}_L$}

\put(160,50){$\mathbf{J}_R$}

\put(65,15){$\mathfrak{X}(M)^{\ast}$}

\put(170,15){$\mathfrak{X}(S)^{\ast}$}

\put(130,70){\vector(-1, -1){40}}

\put(135,70){\vector(1,-1){40}}

\end{picture}\\
Since the actions commute, using Corollary \ref{finocchio} we obtain the following result.

\begin{lemma}
The momentum maps $\J_R$ and $\J_L$ determine a weak dual pair. Moreover $\mathbf{J}_L$ is invariant under the right action of
$\operatorname{Diff}(S)$ and $\mathbf{J}_R$ is invariant under the
left action of $\operatorname{Diff}(M)$.
\end{lemma}

In the rest of this section we consider the cotangent bundle over the open subset $\Emb(S,M)$ of embeddings. We fix a volume form $\mu $ on $S$
and we identify elements in the regular part of $T^*\Emb(S,M)$ with $1$-form densities $P \mu $, where $P:S \to T^*M$ is a smooth map over an embedding $Q=\pi\o P$, where $\pi:T^*M\to M$ denotes the canonical projection. 
Let $T^*\Emb(S,M)^\times $ denote the open subset of the regular part of $T^*\Emb(S,M)$
consisting of those $1$-form densities $P\mu$ which are everywhere non-zero on $S$,
\ie $P(s) \neq 0$ for all $s\in S$.

The restrictions of the momentum maps to $T^*\Emb(S,M)^\times $ 
will be denoted again by $\J_L$ and $\J_R$.
We show below that the action of $\Diff(S)$ is transitive on level sets of $\mathbf{J}_L$
and that the action of $\Diff(M)$ is transitive on connected components of level sets of $\mathbf{J}_R$.
Thus both actions are infinitesimally transitive and we use Proposition \ref{zero} to prove that the restriction 
of the weak EPDiff dual pair to $T^*\Emb(S,M)^\times $  is a dual pair. 

\begin{proposition}\label{diffs}
The diffeomorphism group $\Diff(S)$ acts transitively from the right on the level sets of the left momentum map $\J_L:T^*\Emb(S,M)^\times \to \X(M)^*$.
\end{proposition}

\begin{proof}
Let $P\mu,P'\mu\in T^*\Emb(S,M)^\times $ be two $1$-form densities in the same level set of $\mathbf{J}_L$, with footpoints $Q,Q'\in\Emb(S,M)$.
This means that the identity
\begin{equation}\label{pxq}
\int_SP\cdot( X \o Q)\mu 
=\int_SP'\cdot( X \o Q')\mu 
\end{equation}
holds for each vector field $X\in\X(M)$. 
Since $P(s)\ne 0$ and $P'(s)\ne 0$ for all $s\in S$, 
we conclude that the embeddings $Q$ and $Q'$ have the same image in  $M$. This implies the existence of a diffeomorphism $\ps\in\Diff(S)$ with $Q'=Q\o\ps$. Indeed, it suffices to choose $ \psi := (Q') ^{-1} \circ Q$, where $Q$ is considered as a diffeomorphism $Q:S\rightarrow Q(S)$ and where $ (Q') ^{-1}$ denotes the inverse of the diffeomorphism $Q':S \rightarrow Q(S)=Q'(S)$.

Denoting by $J(\ps)\in\F(M)$ the Jacobian function of $\ps$ with respect to the volume form $\mu$, \ie $\ps^*\mu=J(\ps)\mu$,
we compute
\[
\int_SP\!\cdot\!( X \o Q)\mu=\int_S(P\o\ps)\!\cdot\!(X\o Q')\ps^*\mu
=\int_S(P\o\ps)J(\ps)\!\cdot\!(X\o Q')\mu,
\]
and we rewrite the identity \eqref{pxq} as
\[
\int_S\left((P\o\ps)J(\ps)-P'\right)\!\cdot\!(X\o Q')\mu.
\]
Plugging in a distributional vector field (current) 
$X=v_x\de_x$ with support in 
$x=Q'(s)\in M$ for $s\in S$, where $v_x\in T_xM$,
we obtain $P'(s)=P(\ps(s))J(\ps)(s)$.
Since $s\in S$ was arbitrary, this implies $P'\mu=(P\o\ps)\ps^*\mu=\ps\cdot(P\mu)$,
showing the transitivity of the $\Diff(S)$ action on level sets of $\mathbf{J}_L$.
\end{proof}

\medskip

Let $\Diff(M,N)$ be the group of diffeomorphisms of $M$ fixing pointwise the submanifold $N$ of $M$.
It is a Lie group with Lie algebra $\X(M,N)$,
the Lie algebra of all vector fields vanishing at points of $N$.
The group $\Diff(M,N)$ acts from the left on the vector space of $1$-forms on $M$ along $N$, 
\ie on the space of sections $\Ga(T^*M|_N)$, by
$(\ph\cdot P)(x)=P(x)\o T_x\ph^{-1}$. This makes sense 
since $\ph(x)=x$ for all $x\in N$.
  
The infinitesimal action of $\X(M,N)$ on $\Ga(T^*M|_N)$
can be written as
\begin{equation}\label{nab}
X\cdot P=- \pounds _XP= -\nabla _X P- P\o\nabla X=- P\o\nabla X,
\end{equation}
where $\nabla X$ is any covariant derivative of the vector field $X$, all the points in $N$ being zeros of $X$.

\begin{remark}\label{egal}
{\rm Note that if $P\in\Ga(T^*M|_N)$, then $P\mu\in T_i^*\Emb(N,M)$, where $i:N\to M$ the inclusion and $\mu$ a fixed volume form on $N$. It is easy to see that
the action of the subgroup $\Diff(M,N) \subset \Diff(M)$ introduced above is the restriction 
of the cotangent lifted action of $\Diff(M)$ on $T^*\Emb(N,M)$.
}
\end{remark}

Given a $1$-form $\al\in\Om^1(N)$, we consider the affine subspace
\[
\Ga_\al(T^*M|_N):=\{P\in\Ga(T^*M|_N) : P|_{TN}=\al\}\subset \Ga(T^*M|_N),
\]
with directing linear subspace the conormal bundle to the submanifold $N$ of $M$
\[
\Ga_0(T^*M|_N)=\{P\in\Ga(T^*M|_N):P|_{TN}=0\}.
\]
One observes that $\Ga_\al(T^*M|_N)$
is an invariant submanifold for the $\Diff(M,N)$ action on $\Ga(T^*M|_N)$. Another invariant submanifold is the open subset
$\Ga(T^*M|_N)^\times $ of sections in $\Ga(T^*M|_N)$ which are everywhere on $N$ non-zero. Therefore, their intersection
\[
\Ga_\al(T^*M|_N)^\times: =\left\{P\in\Ga(T^*M|_N)^\times : P|_{TN}=\al\right\},
\]
is also a $\Diff(M,N)$ invariant submanifold of $\Ga(T^*M|_N)$ and we have the following result.

\begin{lemma}\label{hard}
For any $\al\in\Omega^1(N)$, the action of the group $\Diff(M,N)$ on the manifold $\Ga_\al(T^*M|_N)^\times$ is infinitesimally transitive. Moreover, the action is transitive on connected components of $\Ga_\al(T^*M|_N)^\times$.
\end{lemma}

\begin{proof} We fix a Riemannian metric $M$ and denote by $ \nabla $, $\sharp$, and $\|\,\|$ the Levi-Civita covariant derivative, the sharp operator, and the norm associated to the Riemannian metric, respectively.
We identify the tangent space of the affine subspace $\Ga_\al(T^*M|_N)$ at an arbitrary point 
$P\in\Ga_\al(T^*M|_N)^\times $
with the conormal bundle $\Ga_0(T^*M|_N)$.

Taking into consideration the infinitesimal action \eqref{nab}, 
we need to show that given $P\in\Ga_\al(T^*M|_N)^\times $, for every 
$P'\in\Ga_0(T^*M|_N)$ there exists $X\in\X(M,N)$ such that 
$P'=P\o\nabla X$. We will find the vector field in the form $X=fY$, where $f\in\F(M)$ vanishes on $N$, and $Y\in\X(M)$ restricted to the submanifold $N$ is $P^\sharp$.

Consider the section 
\[
\la:=\frac{P'}{\| P\|^2}\in\Ga_0(T^*M|_N),
\]
where $\| P(x)\|^2\ne 0$ because $P(x)\ne 0$
for all $x\in N$.
Lemma \ref{function} can be applied to the restriction of $\la$ to the normal bundle over the submanifold $N\subset M$ 
(viewed as a tubular neighborhood of $N$)
because the restriction of $\la$ to $TN$ vanishes.
We obtain from $\la$ a function $f$ on the normal bundle $E$, linear on each fiber, with the property that
the differential of $f$ along $N$ is $\la$,
\ie $(\dd f)|_{N}=\la$ as sections of $T^*M|_{N}$.
Using the tubular neighborhood $E$ of $N$ we build a smooth
function on $M$, identical to $f$ on a neighborhood of $N$, also denoted by $f$.

We define $X=fY\in\X(M,N)$.
The computation
\[
(\nabla X)|_N=(f\nabla Y)|_N+((\dd f) Y)|_N
=(\dd f)|_N Y|_N=\la P^\sharp
\]
implies that 
\[
P\o (\nabla X)|_N=P\o(\la P^\sharp)=\la\| P\|^2=P'.
\]
This ensures the infinitesimal transitivity of the $\X(M,N)$ action
on $\Ga_\al(T^*M|_N)^\times $.

For the transitivity we consider $P_0$ and $P_1$ in the same connected component of $\Ga_\al(T^*M|_N)^\times $.
Let $P_t$ be a curve in $\Ga_\al(T^*M|_N)^\times $ connecting them, and let $ \dot P_t\in \Ga_0(T^*M|_N)$ be 
its derivative. From the infinitesimal transitivity above we get $X_t\in\X(M,N)$ with $\dot  P_t=X_t\cdot P_t$.
Since the construction of the vector field involves only the Riemannian metric, tubular neighborhoods and extensions,
the vector fields $X_t$ can be chosen smoothly dependent on $t$. 
Let $\ph_t\in\Diff(M,N)$ be the isotopy starting at the identity defined by the time dependent vector field $X_t\in\X(M,N)$.
Then $P_t=\ph_t\cdot P_0$, in particular $P_1=\ph_1\cdot P_0$, and this solves the transitivity.
\end{proof}

\medskip

Using this lemma, we obtain below the analogue of Proposition \ref{diffs} for the right momentum map.

\begin{proposition}\label{diffm}
The left action of the diffeomorphism group $\Diff(M)$ on connected components of the level sets of the right momentum map 
$\J_R:T^*\Emb(S,M)^\times \to\X(S)^*$ is transitive. In particular the action is infinitesimally transitive 
on the level sets of $\J_R$.
\end{proposition}
\begin{proof}
Let $P_t$ be a smooth path connecting two arbitrary elements $P_0$ and $P_1$ in the same connected component of a level set of $\J_R$. Let $Q_t=\pi\o P_t\in\Emb(S,M)$.
By the transitivity of the $\Diff(M)$ action on connected components of the manifold of embeddings, there exists a smooth path $\ph_t\in\Diff(M)$ such that $Q_t=\ph_t\o Q_0$ \cite{H76}. Hence all $\ph_t^{-1}\cdot P_t\in T^*\Emb(S,M)^\times$ have the same footpoint $Q_0$.
Moreover, since $\J_R$ is invariant under the $\Diff(M)$ action, the smooth path $\ph_t^{-1}\cdot P_t$ connects $P_0$ and $P_2=\ph_1^{-1}\cdot P_1$ in the same level set of $\J_R$.
This means that for any vector field $ X \in\X(S)$,
\begin{equation}\label{cc}
\int_S P_0 \cdot (TQ_0\o X )\mu 
=\int_S P_2 \cdot (TQ_0\o X )\mu .
\end{equation}

Elements in $T_{Q_0}^*\Emb(S,M)^\times$ can be viewed as sections in 
$\Ga(T^*M|_N)$ for $N=Q_0(S)$. We conclude from \eqref{cc} that $P_0|_{TN}=P_2|_{TN}=:\al\in\Om^1(N)$, so 
$P_0,P_2\in\Ga_\al(T^*M|_N)^\times $.
They belong to the same connected component of $\Ga_\al(T^*M|_N)^\times $ because 
we can write $\ph_t^{-1}\cdot P_t$ instead of $P_2$ in \eqref{cc}.
By the transitivity property in  Lemma \ref{hard}, we get a diffeomorphism $\ph_2\in\Diff(M,N)$ 
such that $P_2=\ph_2\cdot P_0$.
Now $P_1=\ph_1\cdot P_2=(\ph_1\o\ph_2)\cdot P_0$ shows the requested transitivity of the $\Diff(M)$ action.
\end{proof}

\medskip

Using the transitivity results obtained in Propositions \ref{diffs} and \ref{diffm}, together with Proposition \ref{zero}, we obtain below the main result of this section.

\begin{theorem}\label{homa}
The restriction of the weak dual pair for EPDiff equation
to the open subset $T^*\Emb(S,M)^\times $ of the regular part of $T^*\Emb(S,M)$, 
with canonical symplectic form, is a dual pair.
\end{theorem}


\section{A map between dual pairs}\label{six}

We now turn our attention to the special case when the symplectic manifold involved in the ideal fluid dual pair is a cotangent bundle endowed with the canonical symplectic form. Recall from Remark \ref{45} that in this case, since the symplectic form is exact, the central extension of $\mathfrak{X}_{ex}(S, \mu )$ is not needed to get a well defined momentum map $\mathbf{J}_R: \mathcal{F} (S, T^*M)\rightarrow  \mathfrak{X}_{ex}(S, \mu )^*$.
We will show that in this case there is a natural symplectic map that relates both the Euler and the EPDiff dual pair constructions.

Given a manifold $S$ with volume form $\mu $, we have the natural vector bundle automorphism
\begin{equation}\label{phii}
\Phi: \mathcal{F}(S,T^*M)\rightarrow T^*\mathcal{F}(S,M),\quad\Phi( P ):={P}\mu ,
\end{equation}
covering the identity on $\mathcal{F}(S,M)$, whose image is the regular part of $T^*\F(S,M)$. 
Denoting by $\pi:T^*M\to M$ the cotangent bundle projection, we have $P \mu \in T_Q^*\mathcal{F}(S,M)$ where $Q=\pi\o P$. The pairing with a tangent vector $V_Q$ in $T_Q \mathcal{F}(S,M)=\Ga(Q^*TM)$ is given by
\[
\langle P \mu ,V_Q\rangle=\int_S P(V_Q) \mu .
\]

On $T^*\F(S,M)$ we consider the canonical symplectic form $\Om= -\dd\th_{\F(S,M)}$, where $\th_{\F(S,M)}$ is the canonical $1$-form on the cotangent bundle $T^*\F(S,M)$. On $\mathcal{F}(S,T^*M)$ we take the symplectic form $\bar\om=-\dd\bar\th_M$
determined via the bar map from the canonical symplectic form $\om=-\dd\th_M$ on $T^*M$, where $\th_M$ is the canonical $1$-form on the cotangent bundle $T^*M$.

\begin{lemma}\label{symp}
The vector bundle automorphism $\Phi: (\mathcal{F}(S,T^*M),\bar\om)\rightarrow (T^*\mathcal{F}(S,M),\Om)$
is a symplectic map.
\end{lemma}

\begin{proof}
It is enough to check that $\Ph^*\th_{\F(S,M)}=\bar\th_M$. For this let $P\in \mathcal{F}(S,T^*M)$ and ${\mathcal {V}}_P$ a tangent vector at $P$, \ie a vector field on $T^*M$ along $P$. Then
\begin{align*}
(\Ph^*\th_{\F(S,M)})({\mathcal {V}}_P)&=\th_{\F(S,M)}(T\Ph({\mathcal {V}}_P))
=\left\langle\Ph(P),T(\pi_{\F(S,M)}\o \Ph)({\mathcal {V}}_P)\right\rangle\\
&=\int_SP(T\pi\o {\mathcal {V}}_P)\mu 
=\int_S\th_M({\mathcal {V}}_P)\mu =\bar\th_M({\mathcal {V}}_P),
\end{align*}
where $\pi:T^*M\to M$ and $\pi_{\F(S,M)}:T^*\F(S,M)\to\F(S,M)$ are canonical projections, so $(\pi_{\F(S,M)}\o\Ph)(P)=\pi\circ P$.
\end{proof}

\medskip
The relation between the ideal fluid dual pair and the EPDiff dual pair is presented in the next theorem.

\begin{theorem}\label{commutative_diagram}
The weak dual pair $(\hat\J_L,\J_R)$ for Euler equation involving a cotangent bundle 
and the weak dual pair $(\J_L,\J_R)$ for the EPDiff equation are related by the commutative diagram
\[
\xymatrix{
\F(T^*M)^*\ar[d] &\mathcal{F}(S,T^*M) \ar[l]_{\hat\J_L}\ar[r]^{\J_R} \ar[d]_{\Phi}&
{\X_{ex}(S,\mu)}^*\\
\X(M)^*& T^*\F(S,M)\ar[l]_{\mathbf{J}_L} \ar[r]^{\mathbf{J}_R} &
\X(S)^*\ar[u],\\
}
\]
where the linear map $\X(S)^*\to{\X_{ex}(S,\mu)}^*$ is dual to the inclusion $\X_{ex}(S,\mu)\to\X(S)$
and the linear map $\F(T^*M)^*\to\X(M)^*$
is dual to the canonical map 
\begin{equation}\label{calp}
\mathcal{P}:\X(M)\to\F(T^*M),\quad\mathcal{P}( X )(\alpha_x):=\langle\alpha_x, X (x)\rangle.
\end{equation}
Moreover, all arrows in this diagram are formally Poisson maps.
\end{theorem}

\begin{proof}
Let $P\in\mathcal{F}(S,T^*M)$, $ X \in\X(M)$ and $X_\al\in\X_{ex}(S,\mu)$.
The commutativity of the left diagram follows from the equalities
\begin{align*}
\left\langle\mathbf{J}_L(\Ph(P)), X \right\rangle&=\left\langle\mathbf{J}_L(P \mu ), X\right \rangle\stackrel{\eqref{trei}}{=}\int_S P( X \o Q) \mu \stackrel{\eqref{calp}}{=}\int_S(\P( X )\o P) \mu \\
&\stackrel{}{=}\left\langle\hat\J_L(P),\mathcal{P}( X )\right\rangle.
\end{align*}
The commutativity of the right diagram follows from the equalities
\begin{align*}
\left\langle\mathbf{J}_R(\Ph(P)),X_\al\right\rangle&=\left\langle\mathbf{J}_R(P \mu ),X_\al\right\rangle\stackrel{\eqref{patru}}{=}\int_S P(TQ\o X_\al) \mu \\
&=\int_S P(T\pi\o TP\o X_\al) \mu 
=\int_S \th_M(TP\o X_\al) \mu \\
&
=\int_S \ii_{X_\al}(P^*\th_M)\mu 
=\int_S P^*\th_M\wedge \mathbf{d}\al=-\int_S P^*\om\wedge\al
\stackrel{\eqref{exactj}}{=}\left\langle\J_R(P),X_\al\right\rangle.
\end{align*}

The horizontal maps in the diagram are formally Poisson maps, since they are equivariant momentum maps. The vertical
maps, except $\Phi$, are dual to Lie
algebra homomorphisms ($\mathcal{P}$ is a Lie algebra homomorphism), so they are Poisson maps.
That $\Phi$ is Poisson follows from Lemma \ref{symp}.
\end{proof}

\begin{remark}[compatibility of the actions]
\rm
Since the map $\Phi$ is a symplectic diffeomorphism, one can write the EPDiff dual pair on
$\F(S,T^*M)$ and obtain the diagram

\begin{picture}(150,100)(-70,40)%
\put(110,75){$\F(S,T^*M)$}

\put(77,60){$\mathbf{J}_L$}

\put(185,60){$\mathbf{J}_R$}

\put(60,35){$\mathfrak{X}(M)^{\ast}$}

\put(190,35){$\mathfrak{X}(S)^{\ast}$.}

\put(77,90){$\hat\J_L$}

\put(185,90){$\mathbf{J}_R$}

\put(55,115){$\F(T^*M)^{\ast}$}

\put(180,115){$\mathfrak{X}_{ex}(S,\mu)^{\ast}$}

\put(100,67){\vector(-1, -1){20}}

\put(175,67){\vector(1,-1){20}}

\put(100,90){\vector(-1, 1){20}}

\put(175,90){\vector(1,1){20}}

\put(65,109){\vector(0,-1){60}}

\put(207,49){\vector(0,1){60}}

\end{picture}\\

\noindent Since both dual pairs are now defined on the same space, we can ask whether or not the left (resp. right) actions involved in both dual pairs are compatible. It is clear that the right action of $\Diff_{ex}(S,\mu)$ associated to the ideal fluid dual pair is compatible with the right action of $\Diff(S)$ associated to the EPDiff equation. It is worth pointing out that the same compatibility holds for the left actions.
Indeed, the cotangent-lifted left action of
$\Diff(M)$ on the regular cotangent bundle $\mathcal{F} (S,T^*M)$ of $ \mathcal{F} (S,M)$ 
yields the action of a subgroup of canonical
transformations of $T^*M$. Thus,
the group action producing the left-leg of the EPDiff dual pair
can be considered as the action of a subgroup of
Hamiltonian diffeomorphisms, which in turn produces the left-leg
momentum map ideal fluid dual pair ${\bf J}_L$.
\end{remark}

\appendix

\section{Appendix}

\subsection{The hat calculus}

Differential forms on the Fr\'echet manifold $\F(S,M)$ of smooth functions on a compact $k$--dimensional manifold $S$ can be obtained in a natural way from pairs of differential forms on $M$ and $S$
by the hat pairing. In this Appendix we present the functorial properties of the hat pairing \cite{Vi09}.

\paragraph{Hat pairing.}
Let $\ev:S\x\F(S,M)\to M$ be the \textit{evaluation map} 
$\ev(s,f)=f(s)$, and $\pr:S\x\F(S,M)\to S$ the projection 
on the first factor.
A pair of differential forms $\om\in\Om^p(M)$ and $\al\in\Om^q(S)$,
with $p+q\geq k$, determines a differential form $\widehat{\om\cdot\al}$ on $\F(S,M)$ by the fiber integral over $S$ of the $(p+q)$-form $(\ev^*\om\wedge\pr^*\al)$ on $S\x\F(S,M)$:
\begin{equation}\label{use}
{\widehat{\om\cdot\al}=\fint_S\ev^*\om\wedge\pr^*\al}.
\end{equation}
In this way one obtains a bilinear map called the {\it hat pairing}: 
\begin{equation*}\label{pair}
\Om^p(M)\x\Om^q(S)\to\Om^{p+q-k}(\F(S,M)),\quad (\omega,\alpha)\mapsto \widehat{\om\cdot\al}.
\end{equation*}
For example, the hat pairing of differential forms on $M$ with the constant function $1\in\Om^0(S)$ is
the transgression map
$\om\in\Om^p(M)\to\fint_S\ev^*\om\in\Om^{p-k}(\F(S,M))$.

\begin{remark}[Fiber integration] \rm Recall that, given a compact $k$--dimensional manifold $S$ and a manifold $P$, the fiber integration over $S$ 
assigns to $\beta\in\Om^n(S\x P)$, $n\geq k$, the differential form $\fint_S\beta\in\Om^{n-k}(P)$ defined by
\[
\left(\fint_S\beta\right)(p)=\int_S\beta_p\,\in\,\La^{n-k}T^*_pP,\quad\text{for all}\quad p\in P,
\]
where the $\La^{n-k}T^*_pP$-valued $k$-form $\beta_p\in\Om^k\left(S,\La^{n-k}_pP\right)$ is defined by
\[
\left\langle\beta_p(s)\left(u_s^1,\dots,u_s^{k}\right),v_p^1\wedge\dots\wedge v_p^{n-k}\right\rangle:=
\beta(s,p)\left((0_s,v_p^1),\dots,(0_s,v_p^{n-k}),(u_s^1,0_p),\dots,(u_s^{k},0_p)\right)
\]
for all $u_s^j\in T_sS$ and $v_p^i\in T_pP$. We refer to \cite{GHV72} for more informations about fiber integration. Note that in formula \eqref{use}, $P$ is the infinite dimensional manifold $\mathcal{F}(S,M)$ and $\beta$ is the $(p+q)$-form $(\ev^*\om\wedge\pr^*\al)\in\Omega^{p+q}\left(S\times\mathcal{F}(S,M)\right)$.
\end{remark}

An explicit expression of the hat pairing avoiding fiber integration is
\begin{align}\label{ffff}
(\widehat{\om\cdot\al})(f)\left(U_f^1,\dots,U_f^{p+q-k}\right)
=\int_Sf^*\left(\mathbf{i}_{U_f^{p+q-k}}\dots \mathbf{i}_{U_f^1}(\om\o f)\right)\wedge\al,
\end{align}
for $U_f^1,\dots U_f^{p+q-k}\in T_f\F(S,M)$.
Here we denote by $f^*\be_f$ the ``restricted pull-back" by $f$ of a section $\be_f$ of $f^*(\La^mT^*M)\rightarrow S$. Thus, $f^*\beta_f$ is the differential $m$-form on $S$ given by
\[
(f^*\be_f)(s)\left(u^1_s,\dots,u^m_s\right)=\beta_f(s)\left(T_sf(u^1_s),\dots,T_sf(u^m_s)\right),\quad\text{for all}\quad u^i_s\in T_sS.
\]
The fact that \eqref{use} and \eqref{ffff} provide the same differential form on $\F(S,M)$ can be deduced from the identity
\begin{align*}
&(\ev^*\om){(s,f)}\left(\left(0_s,U_f^1\right),\dots, \left(0_s,{U_f^{p-k}}\right),\left(u_s^1,0_f\right),\dots,\left(u_s^k,0_f\right)\right)\\
&\qquad\qquad\qquad =\left(f^*\left(\mathbf{i}_{U_f^{p-k}}\dots \mathbf{i}_{U_f^1}(\om\o f)\right)\right)(s)\left(u_s^1,\dots, u_s^k\right)
\end{align*}
for $U_f^1,\dots,U_f^{p-k}\in T_f\F(S,M)$ and $u_s^1,\dots,u_s^k\in T_sS$.

In the particular case when $f$ is an embedding, $f\in\operatorname{Emb}(S,M)$, then formula \eqref{ffff} simplifies. Indeed, since the action of $\X(M)$ is infinitesimally transitive on the open subset $\Emb(S,M)\subset\F(S,M)$ of embeddings, we can express 
$\widehat{\om\cdot\al}$ at $f\in\Emb(S,M)$ as
\begin{equation}\label{embf}
(\widehat{\om\cdot\al})(f)\left(U_1\o f,\dots, U_{p+q-k}\o f\right)=\int_Sf^*\left(\mathbf{i}_{U_{p+q-k}}\dots \mathbf{i}_{U_1}\om\right)\wedge\al,
\end{equation}
for all $U_1,\dots,U_{p+q-k}\in\X(M)$.

The next proposition follows from the known fact that differentiation and fiber integration along a boundary free manifold $S$ commute.

\begin{proposition}\label{dddd}
The exterior derivative $\dd$ 
is a derivation for the hat pairing, \ie
\begin{equation}\label{deri}
\dd(\widehat{\om\cdot\al})= \widehat{(\dd\om)\cdot\al}+(-1)^{p}\widehat{\om\cdot\dd\al},
\end{equation}
where $\om\in\Om^p(M)$ and $\al\in\Om^q(S)$.
In particular the hat pairing induces a bilinear map for de Rham cohomology spaces
$$
H^p(M)\x H^q(S)\to H^{p+q-k}(\F(S,M)).
$$ 
\end{proposition}

\paragraph{Left $\Diff(M)$ action.}
The natural left action of the group of diffeomorphisms $\Diff(M)$ on $\F(S,M)$ 
is $\bar \ph(f)=\ph\o f$. The associated infinitesimal action of $X\in\X(M)$ is the vector field $\bar X$ on $\F(S,M)$ given by
\[
\bar X(f)=X\o f,\quad\text{for all}\quad  f\in\F(S,M).
\] 
The Jacobi-Lie bracket of two infinitesimal generators is 
$\left[\bar X,\bar Y\right]=\overline{[X,Y]}$. 

\begin{proposition}\label{pppp}
Given $\om\in\Om^p(M)$ and $\al\in\Om^q(S)$, the identities
\begin{enumerate}
\item $\bar\ph^*\widehat{\om\cdot\al}=\widehat{((\ph^*\om)\cdot\al)}$
\item $\pounds_{\bar X}\widehat{\om\cdot\al}=\widehat{((\pounds_X\om)\cdot\al)}$
\item $\mathbf{i}_{\bar X}\widehat{\om\cdot\al}=\widehat{((\mathbf{i}_X\om)\cdot\al)}$ 
\end{enumerate}
hold for all diffeomorphisms $\ph\in\Diff(M)$ and vector fields $X\in\X(M)$.
\end{proposition}


\paragraph{Right $\Diff(S)$ action.}
The natural right action of the diffeomorphism group $\Diff(S)$ on $\F(S,M)$ is
$\hat\ps(f)=f\o\ps$. The infinitesimal action of $Z\in\X(S)$ is the vector field $\hat Z$ on $\F(S,M)$ given by
\[
\widehat Z(f)=Tf\o Z,\quad\text{for all}\quad  f\in\F(S,M).
\]
The Jacobi-Lie bracket of two infinitesimal generators is 
$[\hat X,\hat Y]=-\widehat{[X,Y]}$. 

\begin{proposition}\label{five}
Given $\om\in\Om^p(M)$ and $\al\in\Om^q(S)$, the identities 
\begin{enumerate}
\item $\widehat\ps^*\widehat{\om\cdot\al}=\widehat{(\om\cdot(\ps^{-1})^*\al)}$
\item $\pounds_{\widehat Z\,}\widehat{\om\cdot\al}=-\widehat{(\om\cdot \pounds_Z\al)}$
\item $\mathbf{i}_{\widehat Z\,}\widehat{\om\cdot\al}=(-1)^{p+1}\widehat{(\om\cdot \mathbf{i}_Z\al)}$ 
\end{enumerate}
hold for all diffeomorphisms $\ps\in\Diff(S)$ and vector fields $Z\in\X(S)$.
\end{proposition}


\paragraph{Bar map.}
When a volume form $\mu$ on the compact $k$--dimensional manifold $S$ is given, the \textit{bar map} associates to each differential $p$-form $\omega$ on $M$ a differential $p$-form $\bar\omega$ on $\F(S,M)$ defined by
\[
\bar\om(f)\left(U^1_f,\dots,U^p_f\right)=\int_S \om\left(U^1_f,\dots,U^p_f\right)\mu,
\quad\text{for all}\quad U^i_f\in T_f \mathcal{F}(S,M),
\]
where $\om\left(U^1_f,\dots,U^p_f\right):s\mapsto \om_{f(s)}\left(U_f^1(s),\dots, U_f^p(s)\right)$ defines a smooth function on $S$.
Formula \eqref{ffff} ensures that this bar map is just the hat pairing of differential forms on $M$ with the volume form $\mu$ on $S$
\begin{equation}\label{barb}
\bar\om=\widehat{\om\cdot\mu}=\fint_S\ev^*\om\wedge\pr^*\mu.
\end{equation}

>From the properties of the hat pairing presented in Propositions \ref{pppp} and \ref{dddd}, one obtains the following properties of the bar map.

\begin{proposition}\label{cor2}
For any $\om\in\Om^p(M)$, $\ph\in\Diff(M)$ and $X\in\X(M)$, the following identities hold: 
\begin{enumerate}
\item $\bar\ph^*\bar{\om}=\overline{\ph^*\om}$
\item $\pounds_{\bar X}\bar\om=\overline{\pounds_X\om}$
\item $\mathbf{i}_{\bar X}\bar\om=\overline{\mathbf{i}_X\om}$
\item $\dd\bar\om=\overline{\dd\om}$.
\end{enumerate}
\end{proposition}


\subsection{Momentum maps and nonequivariance}\label{momap_noneq}

We quickly review here some properties of momentum maps in order to fix our conventions (see \cite{MaRa99}). Let $G$ be a Lie group acting canonically on a symplectic manifold $(M,\omega)$. We will denote by $m\mapsto g\cdot m$, (resp. $m\mapsto m\cdot g$) a left, (resp. right) action of $G$ on $M$. Let $\xi_M$ be the infinitesimal generator associated to the Lie algebra element $\xi\in\mathfrak{g}$. We say that the action admits a momentum map, denoted $\mathbf{J}:M\rightarrow\mathfrak{g}^*$, if $\mathbf{J}$ verifies
\begin{equation}\label{def_momentum_map}
\mathbf{d}\langle\mathbf{J},\xi\rangle=\mathbf{i}_{\xi_M}\omega,\quad\text{for all}\quad\xi\in \mathfrak{g},
\end{equation}
where $\langle\mathbf{J},\xi\rangle$ denotes the function on $M$ defined by $\langle\mathbf{J},\xi\rangle(m)=\langle\mathbf{J}(m),\xi\rangle$. 

\paragraph{Connected symplectic manifold.}
When $M$ is connected, the \textit{nonequivariance  cocycle} $c:G\rightarrow\mathfrak{g}^*$ is defined by
\[
c(g):=\mathbf{J}(g\cdot m)-\operatorname{Ad}^*_{g^{-1}}\mathbf{J}(m),\quad\text{resp.}\quad c(g):=\mathbf{J}(m\cdot g)-\operatorname{Ad}^*_g\mathbf{J}(m),
\]
where $m\in M$ can be chosen arbitrarily. Taking the derivative of $c$ at the identity, one obtains the Lie algebra 2--cocycle $\sigma:\mathfrak{g}\times\mathfrak{g}\rightarrow \mathbb{R}$,
\begin{equation}\label{def_noneq_cocycle}
\sigma(\xi,\eta)=\left\langle\mathbf{J},[\xi,\eta]\right\rangle-\omega\left(\xi_M,\eta_M\right),\quad\text{resp.}\quad\sigma(\xi,\eta)=-\left\langle\mathbf{J},[\xi,\eta]\right\rangle-\omega\left(\xi_M,\eta_M\right).
\end{equation}
The cohomology class of the nonequivariance cocycles $c$ and $\si$
does not depend on the choice of the momentum map.

The momentum map $\mathbf{J}$ is said to be \textit{equivariant}  when $c=0$, and \textit{infinitesimally equivariant} when $\sigma=0$. When $\mathbf{J}$ is infinitesimally equivariant, then it is a Poisson map relative to the symplectic bracket on $M$ and the $(+)$ (resp. $(-)$) Lie-Poisson bracket on $\mathfrak{g}^*$ given by
\[
\{f,g\}_{\pm}=\pm\left\langle\mu,\left[\frac{\delta f}{\delta\mu},\frac{\delta f}{\delta\mu}\right]\right\rangle,\quad f,g\in\mathcal{F}(\mathfrak{g}^*).
\]

>From \eqref{def_momentum_map} we see the momentum map depends only on the Lie algebra action. Note that if there is no underlying Lie group action, we can still define the 2--cocycle $\sigma$ and the concept of infinitesimal equivariance is still well-defined.

The 2--cocycle $\sigma$ can be used to produce an infinitesimally equivariant momentum map. It suffices to consider the Lie algebra central extension $\hat{\mathfrak{g}}=\mathfrak{g}\oplus\mathbb{R}$ of $\mathfrak{g}$, with Lie bracket $[(\xi,a),(\eta,b)]=([\xi,\eta],\sigma(\xi,\eta))$. Indeed, the induced Lie algebra action of $\hat{\mathfrak{g}}$ on $M$ admits the infinitesimally equivariant momentum map $\hat{\mathbf{J}}:M\rightarrow\hat{\mathfrak{g}}^*$ given by
\begin{equation}\label{hat_momap}
\hat{\mathbf{J}}(m)=(\mathbf{J}(m),-1),\quad\text{resp.}\quad \hat{\mathbf{J}}(m)=(\mathbf{J}(m),1).
\end{equation}

\paragraph{Nonconnected symplectic manifold.}
We consider the case of a Hamiltonian left action,
the right actions can be treated in a similar way.
There is a $H^0(M)$-valued Lie algebra cocycle $\si$ on $\g$ measuring the nonequivariance of the momentum map:
\begin{equation}\label{sig}
\sigma(\xi,\eta)(m)=\left\langle\mathbf{J}(m),[\xi,\eta]\right\rangle-\omega\left(\xi_M,\eta_M\right)(m),\quad\forall m\in M.
\end{equation}
One sees that $\dd(\si(\xi,\et))=0$, so the cocycle $\si$ takes indeed values in $H^0(M)=\ker(\dd:\F(M)\to\Om^1(M))$.

To each point $m\in M$ one associates an element $[m]\in H_0(M)$ 
defined with the help of the duality pairing by
$\langle f,[m]\rangle=f(m)$ for all $f\in H^0(M)$.
Choosing one point in each connected component of $M$,
one gets a basis of $H_0(M)$.

\begin{proposition}\label{ue}
Let $\hat\g$ denote the central extension of $\g$ by $H^0(M)$
defined by $\si$. Its dual is identified with $\g^*\oplus H_0(M)$.
The induced Lie algebra action of $\hat{\mathfrak{g}}$ on $M$ admits the infinitesimally equivariant momentum map 
\begin{equation}\label{none}
\hat\J:M\to\hat\g^*,\quad\hat{\mathbf{J}}(m)=(\mathbf{J}(m),-[m])\in\g^*\oplus H_0(M).
\end{equation}
\end{proposition}

As seen in the previous paragraph, for each connected component of $M$, 
we get a 1-dimensional central extension of $\g$ with equivariant momentum map. 
A characteristic Lie algebra 2-cocycle for this  extension is $\langle\si,[m]\rangle$,
where $m$ is an arbitrary point in this connected component.
There are special cases when it is possible to find ``smaller" central extensions of $\g$
admitting equivariant momentum maps on whole $M$:

\begin{proposition}\label{te}
Assume that the $H^0(M)$-valued non-equivariance cocycle $\si$ can be written in the form
$\si=T\o\si_T$, where $T:V\to H^0(M)$ is a linear map and
$\si_T$ is a $V$-valued 2-cocycle on $\g$.
Let $\hat\g_T$ be the central extension of $\g$ by $V$ defined with the cocycle $\si_T$.
Then the infinitesimal action of $\hat\g_T$ (inherited from $\g$)
admits an equivariant momentum map 
\[
\hat\J_T:M\to\hat\g_T^*=\g^*\oplus V^*,\quad\hat\J_T(m)=\left(\J(m),-T^*([m])\right).
\]
\end{proposition}


{\footnotesize

\bibliographystyle{new}

}

\end{document}